\documentclass[12pt,reqno]{amsart}

\newcommand{\mat}[4]
{\left(
\begin{array}{cc}
#1 & #2 \\
#3 & #4
\end{array}
\right)}

\def\Ga{\Gamma}

\def\Mat{{\rm Mat}}

\def\bean{\begin{eqnarray*}}
\def\eean{\end{eqnarray*}}

\numberwithin{equation}{section}

\usepackage[final, 
color]{showkeys}
\definecolor{refkey}{gray}{.85}
\definecolor{labelkey}{gray}{.85}

\usepackage{AKstyle}

\usepackage[pagebackref=true, colorlinks]{hyperref}

\hypersetup{pdffitwindow=true,linkcolor=blue,citecolor=blue,urlcolor=blue,menucolor=blue}

\usepackage{comment}

\renewcommand{\d}{u}

\begin{document}

\author{Jean Bourgain}
\thanks{Bourgain is partially supported by NSF grant DMS-0808042.}
\email{bourgain@ias.edu}
\address{IAS, Princeton, NJ}
\author[Alex Kontorovich]{Alex Kontorovich\\
With an Appendix by P\'eter P. Varj\'u}
\thanks{Kontorovich is partially supported by  NSF grants DMS-1209373, DMS-1064214 and DMS-1001252.}
\email{alex.kontorovich@yale.edu}
\address{Yale University, New Haven, CT}

\thanks{%
Varj\'u is partially supported by the Simons Foundation and the European Research Council (Advanced Research Grant 267259).}
\email{pv270@dpmms.cam.ac.uk}
\address{University of Cambridge, UK}

\title[Integral Apollonian Gaskets]
{On the Local-Global Conjecture for Integral Apollonian 
Gaskets}

\begin{abstract}
We prove that a set of density one satisfies the
 local-global conjecture for integral Apollonian 
 gaskets.
That is, for a fixed integral, primitive Apollonian gasket, almost every (in the sense of density) admissible (passing local obstructions) integer is the curvature of some circle 
in the gasket. 
\end{abstract}
\date{\today}
\maketitle
\tableofcontents

\newpage

\section{Introduction}

\begin{figure}
\includegraphics[width=
2in]{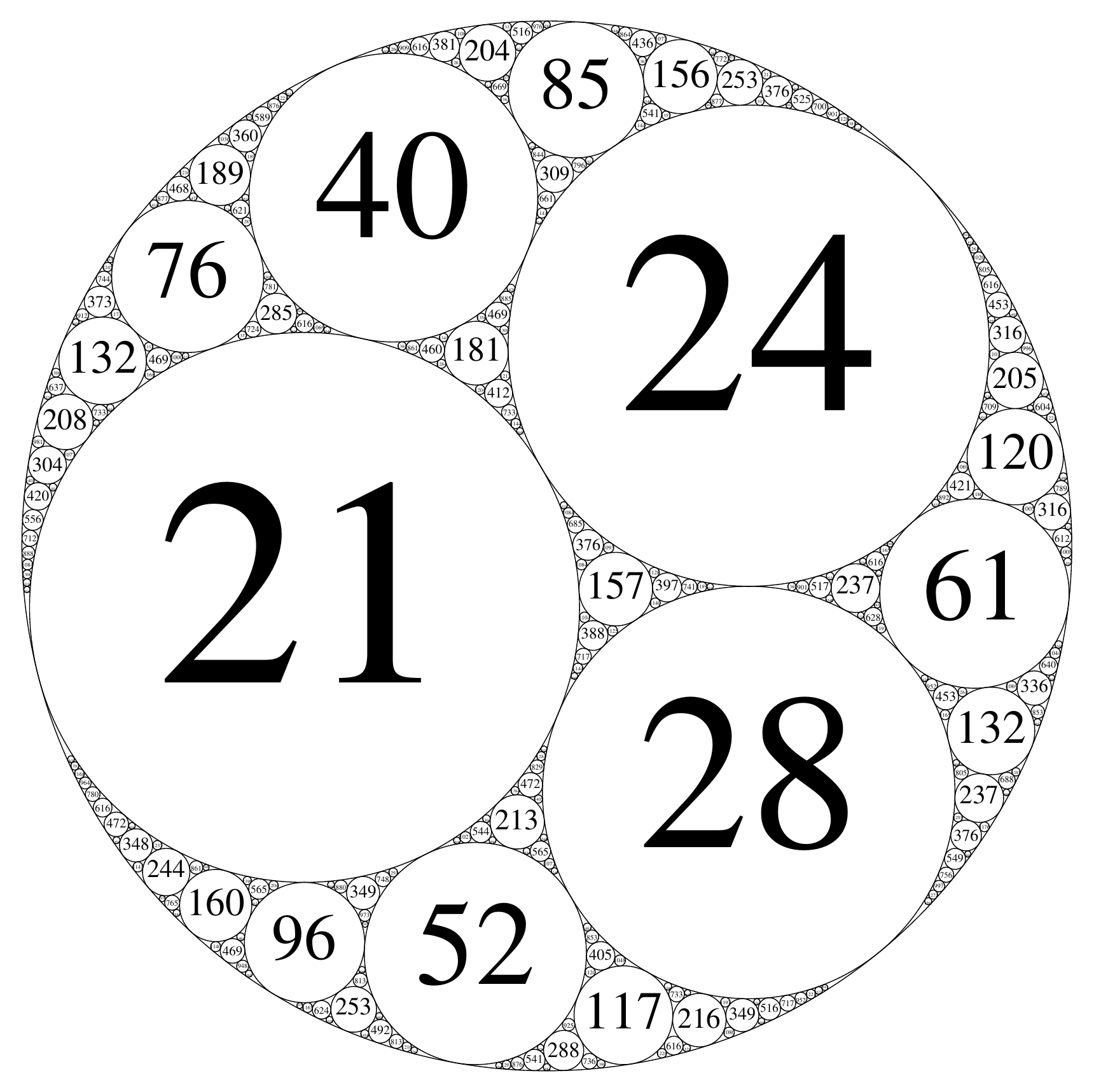}
\vskip-1.25in
\hskip-2.5in
{\Huge -11}
\vskip1in

\caption{The Apollonian gasket with root quadruple $v_{0}=(-11, 21,24,28)^{t}$.}
\label{fig1}
\end{figure}

\subsection{The 
Local-Global
Conjecture}\

Let $\sG$ be an Apollonian 
gasket, see Fig. \ref{fig1}. The number $b(C)$ shown inside a circle $C\in\sG$ is its curvature, that is, the reciprocal of its radius (the bounding circle has 
negative
orientation). Soddy \cite{Soddy1937} first observed the existence of {\it integral} 
gaskets $\sG$, meaning
  ones for which $b(C)\in\Z$ for all $C\in\sG$. Let 
$$
\sB=\sB_{\sG}:=\{b(C):C\in\sG\}
$$ 
be the set of all 
curvatures in $\sG$. We call a 
gasket {\it primitive} if $\gcd(\sB)=1$. From now on, we restrict our attention to a fixed primitive integral Apollonian gasket $\sG$.

Graham, Lagarias, Mallows, Wilks, and Yan \cite{LagariasMallowsWilks2002, GrahamLagarias2003} initiated a detailed study of Diophantine properties of $\sB$, with two separate families of problems (see also e.g. \cite{KontorovichOh2011, FuchsSanden2011, Sarnak2011}): studying $\sB$ with multiplicity (that is, studying circles), or without multiplicity (studying the integers which arise). 
In the present paper, we are concerned with the latter.

In particular,
 the following striking local-to-global conjecture for $\sB$
 is given in 
 \cite[p. 37]{GrahamLagarias2003}, \cite{FuchsSanden2011}.
Let $\sA=\sA_{\sG}$ denote the {\it admissible} integers, that is, those passing
all
 local (congruence) obstructions:
$$
\sA:=\{n\in\Z:n\in\sB(\mod q),\text{ for all $q\ge1$}\}.
$$
\begin{conj}[Local-Global Conjecture]\label{conj}
%
Fix 
 a primitive, integral Apollonian 
gasket $\sG$.
Then every sufficiently large admissible number is the curvature of a circle in $\sG$.
That is, if $n\in\sA$ and $n\gg
1$, then $n\in\sB$.
%
\end{conj}

The purpose of this paper is to prove  the following

\begin{thm}\label{thm:Main}
 Almost every admissible number is the curvature of a circle in 
 $\sG$. Quantitatively, the number of exceptions up to $N$ is bounded by
 $O(N^{1-\eta})$, where $\eta>0$ is effectively computable.
\end{thm}

Admissibility is completely explained in Fuchs's thesis \cite{FuchsThesis}, and is a condition restricting to certain residue classes modulo $24$, cf. Lemma \ref{lem:GGq}.
E.g. for the gasket in Fig. \ref{fig1}, $n\in\sA$ iff 
\be\label{eq:locA}
n\equiv0, 4, 12, 13, 16,\text{ or  }21(\mod 24).
\ee
Thus $\sA$ contains one of every four numbers (six admissible residue classes out of $24$), and Theorem \ref{thm:Main} can be restated in this case as
$$
\#(\sB\cap[1,N])={N\over4}\left(1+O(N^{-\eta})\right).
$$
In general, the local obstructions are easily determined (see Remark \ref{rmk:locA}) from the so-called {\it root quadruple} 
\be\label{eq:v0Def}
v_{0}=v_{0}(\sG)
,
\ee
which is the column vector of the four smallest curvatures in $\sB$. For the gasket in Fig. \ref{fig1}, $v_{0}=(-11, 21, 24, 28)$.
\\

The history of this problem is as follows. The first progress 
towards 
the Conjecture
was already made in \cite{GrahamLagarias2003}, who showed
that
\be\label{eq:GLbnd}
\#(\sB\cap[1,N])\gg N^{1/2}.
\ee
Sarnak \cite{SarnakToLagarias} 
improved this to
\be\label{eq:SarBnd}
\#(\sB\cap[1,N])\gg {N\over (\log N)^{1/2}},
\ee
and then Fuchs \cite{FuchsThesis} 
showed
$$
\#(\sB\cap[1,N])\gg {N\over (\log N)^{0.150\dots}}.
$$
Finally Bourgain and Fuchs \cite{BourgainFuchs2010} settled the so-called ``Positive Density Conjecture,'' that
$$
\#(\sB\cap[1,N])\gg N.
$$

\subsection{Methods}\

Our main approach is through the 
Hardy-Littlewood 
circle method, combining two new ingredients. The first, applied to the major arcs, is effective bisector counting 
in infinite volume 
hyperbolic $3$-
folds, recently achieved by I. Vinogradov \cite{Vinogradov2013}, as well as the uniform spectral gap over congruence towers of such, 
see the Appendix by P\'eter Varj\'u.
The second ingredient is the minor arcs analysis, 
inspired by
that given recently by the first-named author in \cite{Bourgain2012}, where it was proved that the  prime curvatures in a 
gasket
constitute  a positive proportion of the primes.
(Obviously Theorem \ref{thm:Main} implies that 
 $100\%$ of the admissible prime curvatures appear.)

\subsection{Plan for the Paper}\

A more detailed outline of the proof, as well as the setup of some relevant exponential sums, is given in \S\ref{sec:outline}. Before we can do this, we need to recall the Apollonian group 
and some of its subgroups in \S\ref{sec:prelim}. After the outline in \S\ref{sec:outline}, we use \S\ref{sec:preII} to collect some background
 from the spectral and representation theory of infinite volume hyperbolic quotients.
 Then some lemmata are reserved for \S\ref{sec:lems}, the major arcs are estimated in \S\ref{sec:Maj}, and the minor arcs are dealt with in \S\S\ref{sec:qQ0}-\ref{sec:QXT}.
The Appendix, by P\'eter Varj\'u, extracts the spectral gap property for the Apollonian group from that of its arithmetic subgroups.

\subsection{Notation}\

We use the following standard notation. Set $e(x)=e^{2\pi i x}$ and $e_{q}(x)=e(\frac xq)$. We use $f\ll g$ and $f=O(g)$ interchangeably; moreover $f\asymp g$ means $f\ll g\ll f$.
Unless otherwise specified, the implied constants may depend at most on the 
gasket
  $\sG$ (or equivalently on the root quadruple $v_{0}$), which is treated as fixed. 
The symbol $\bo_{\{\cdot\}}$ is the indicator function of the event $\{\cdot\}$. 
The greatest common divisor of $n$ and $m$ is written $(n,m)$,  their least common multiple is $[n,m]$, and $\gw(n)$ denotes the number of distinct prime factors of $n$. The cardinality of a finite set $S$ is denoted $|S|$ or $\# S$.
The transpose of a matrix $g$ is written $g^{t}$. The prime symbol $'$ in $\underset{r(q)}{\gS} {}'$ means the range of $r(\mod q)$ is restricted to $(r,q)=1$. Finally, $p^{j}\| q$ denotes $p^{j}\mid q$ and $p^{j+1}\nmid q$.

\subsection*{Acknowledgements} The authors are grateful to Peter Sarnak for illuminating discussions, and many detailed comments improving the exposition of an earlier version of this paper. We 
thank
Tim Browning, Sam Chow,  Hee Oh, Xin Zhang, and 
the referee for numerous 
corrections and suggestions. 

\newpage

\section{Preliminaries I: The Apollonian Group and Its Subgroups}\label{sec:prelim}


\subsection{Descartes Theorem and Consequences}\


 Descartes' Circle Theorem states that a quadruple $v$
 of (oriented) curvatures of four mutually tangent circles 
  lies on the cone
\be\label{eq:Fv}
F(v)=0,
\ee
where 
$F$ is the Descartes quadratic form:
\be\label{eq:Fdef}
F(a,b,c,d)
=
2(a^{2}+b^{2}+c^{2}+d^{2})
-(a+b+c+d)^{2}
.
\ee
Note that $F$ has signature $(3,1)$ over $\R$,
and let 
$$
G:=\SO_{F}(\R)=\{g\in\SL(4,\R):F(g v)=F(v),\text{ for all }v\in\R^{4}\}
$$ 
be the real special orthogonal group preserving $F$.

It follows immediately that for $b,c$ and $d$ fixed, there are two solutions $a,a'$ to \eqref{eq:Fv}, and
$$
a+a'=2(b+c+d).
$$
Hence we observe that $a$ can be changed into $a'$ by a reflection, that is,
$$
(a,b,c,d)^{t}=S_{1}\cdot (a',b,c,d)^{t},
$$
where the reflections
$$
S_{1}=
\bp
-1&2&2&2\\
&1&&\\
&&1&\\
&&&1
\ep
,\qquad
S_{2}=
\bp
1&&&\\
2&-1&2&2\\
&&1&\\
&&&1
\ep
,
$$
$$
S_{3}=
\bp
1&&&\\
&1&&\\
2&2&-1&2\\
&&&1
\ep
,\qquad
S_{4}=
\bp
1&&&\\
&1&&\\
&&1&\\
2&2&2&-1
\ep
,
$$
generate the  so-called {\it Apollonian group} 
\be\label{eq:GamDef}
\cA
=
\<S_{1},S_{2},S_{3},S_{4}
\>
.
\ee
It is a Coxeter group, free 
except for
the relations $S_{j}^{2}=I$, $1\le j\le 4$. We immediately pass to the 
index two
subgroup
$$
\G:=\cA\cap \SO_{F}
$$
of orientation preserving transformations, that is, even words in the generators.
Then $\G$ is freely generated by $S_{1}S_{2}$, $S_{2}S_{3}$ and $S_{3}S_{4}$.
%
It is known 
that  $\G$ is Zariski dense in  $G$ but {\it thin}, that is,  of infinite index in $G(\Z)$; equivalently, the Haar measure of $\G\bk G$ is infinite.

\subsection{Arithmetic Subgroups%
}\

Now we review the arguments from \cite{GrahamLagarias2003, SarnakToLagarias} which 
lead to
\eqref{eq:GLbnd} and \eqref{eq:SarBnd}, as our setup depends critically on them.

Recall that for any fixed 
gasket
 $\sG$, there is a root quadruple $v_{0}$ of the four smallest 
curvatures in $\sG$, cf. \eqref{eq:v0Def}.
It follows 
 from 
\eqref{eq:Fv} and \eqref{eq:GamDef}
 that 
 the set $\sB$ of all 
 curvatures can be realized as the orbit of the root quadruple $v_{0}$ under $\cA$.
Let 
$$
\sO=\sO_{\sG}:=\G\cdot v_{0}
$$
be the orbit of $v_{0}$ under $\G$.
Then 
the set of all 
curvatures 
certainly contains
\be\label{eq:setB}
\sB
\supset
\bigcup_{j=1}^{4}\<e_{j},\sO\>
=
\bigcup_{j=1}^{4}\<e_{j},\G\cdot v_{0}\>
,
\ee
where $e_{1}=(1,0,0,0)^{t},\dots, e_{4}=(0,0,0,1)^{t}$ constitute the standard basis for $\R^{4}$, and the inner product above is the standard one. Recall we are treating $\sB$ as a set, that is, without multiplicities.
\\

It was observed in \cite{GrahamLagarias2003} that $\G$ contains unipotent elements, and hence one can use these to furnish an injection of affine space in the otherwise intractable orbit $\sO$, as follows.
Note first that 
\be\label{eq:C1def}
C_{1}:=
S_{4}S_{3}
=
\left(
\begin{array}{cccc}
 1 &  &  &  \\
 & 1 &  &  \\
 2 & 2 & -1 & 2 \\
 6 & 6 & -2 & 3
\end{array}
\right)
\in\G
,
\ee
and 
after conjugation by
$$
J:=
\left(
\begin{array}{cccc}
 1 &  &  &  \\
 -1 & 1 &  &  \\
 -1 & 1 & -2 & 1 \\
 -1 &  &  & 1
\end{array}
\right)
,
$$
we have
$$
\tilde C_{1}:=
J^{-1}\cdot C_{1}\cdot J
=
\left(
\begin{array}{cccc}
 1 &  &  &  \\
  & 1 &  &  \\
  & 2 & 1 &  \\
  & 4 & 4 & 1
\end{array}
\right)
.
$$
Recall the spin homomorphism $\rho: 
\SL_{2}\to\SO(2,1)$, embedded for our purposes in $\SL_{4}$, given explicitly by
\be\label{eq:spin}
\rho
:
\mattwo
\ga\gb
\g\gd
\mapsto
{1\over \ga\gd-\gb\g}
\bp
1&&&\\
&{\ga^{2}}&{2\ga\g}&{\g^{2}}\\
&{\ga\gb}&{\ga\gd+\gb\g}&{\g\gd}\\
&{\gb^{2}}&{2\gb\gd}&{\gd^{2}}
\ep
.
\ee
In fact $\SL_{2}$ is a double cover of $\SO(2,1)$ under $\rho$, with kernel $\pm I$. 
It is clear from inspection that 
$$
\rho:
\mattwo1201=:T_{1}
\mapsto 
\tilde C_{1}
.
$$
Since
$T_{1}^{n}=\mattwo1{2n}01$, 
 for each $n\in\Z$,  $\G$ contains the element
$$
C_{1}^{n}=
J\cdot \rho(T_{1}^{n})\cdot J^{-1}
=
\left(
\begin{array}{cccc}
 1 & 0 & 0 & 0 \\
 0 & 1 & 0 & 0 \\
 4 n^2-2 n & 4 n^2-2 n & 1-2 n & 2 n \\
 4 n^2+2 n & 4 n^2+2 n & -2 n & 2 n+1
\end{array}
\right).
$$
(Of course this can be seen directly from \eqref{eq:C1def}; these transformations will be more enlightening below.)

Thus if 
$v=(a,b,c,d)^{t}\in\sO$ is 
a
 quadruple 
in
 the orbit
 , then $\sO$ also contains
$
C_{1}^{n}\cdot v
$
for all $n$.
From \eqref{eq:setB}, we then have that the set $\sB$ of curvatures contains 
\be
\label{eq:oneParab}
\sB\ni\<e_{4},C_{1}^{n}\cdot v\>
=
4( a
+ b) n^2
+2
( a 
+ b 
- c 
+ d) n
+d
.
\ee
The circles thus generated 
are all tangent to two fixed circles, which explains the square 
curvatures in Fig. \ref{fig2}.
Of course \eqref{eq:oneParab} immediately implies \eqref{eq:GLbnd}.
\\

\begin{figure}
\includegraphics[width=2.5in]{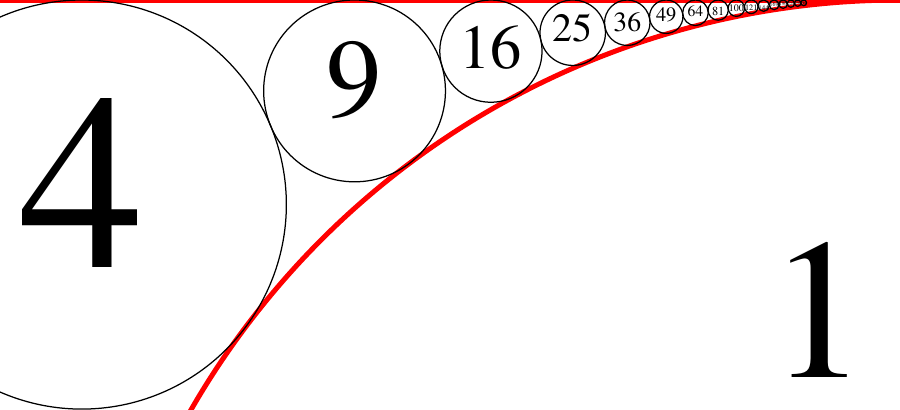}
\caption{Circles tangent to two fixed circles.}
\label{fig2}
\end{figure}

Observe further that 
$$
C_{2}
:=
S_{2}S_{3}
=
\left(
\begin{array}{cccc}
 1 &  &  &  \\
 6 & 3 & -2 & 6 \\
 2 & 2 & -1 & 2 \\
  &  &  & 1
\end{array}
\right)
$$
is another unipotent element, with 
$$
\tilde C_{2}:=
J^{-1}\cdot C_{2}\cdot J
=
\left(
\begin{array}{cccc}
 1 &  &  &  \\
 & 1 & 4 & 4 \\
 &  & 1 & 2 \\
 &  &  & 1
\end{array}
\right)
,
$$
and 
$$
\rho:
\mattwo1021
=:T_{2}
\mapsto 
\tilde C_{2}
.
$$

Since 
 $T_{1}$ and $T_{2}$ generate $\gL(2)$, the principal $2$-congruence subgroup of $\PSL(2,\Z)$, we see that
the Apollonian group $\G$ contains the subgroup 
\be\label{eq:G1}
\Xi:=
\<C_{1},C_{2}\>
=
J\cdot
\rho\bigg(
\gL(2)
\bigg)
\cdot
J^{-1}
<
\G
.
\ee
%
In particular,
whenever $(2x,y)=1$, 
there is an element 
$$
\mattwo*{2x}*{y}\in\gL(2)
,
$$ 
and thus $
\Xi
$ contains
the element
\bea
\label{eq:xiXYdef}
\xi_{x,y}
&:=&
J\cdot
\rho
\mattwo*{2x}*{y}
\cdot
J^{-1}
\\
\nonumber
&=&
\left(
\begin{array}{cccc}
 1 & 0 & 0 & 0 \\
*&*&*&*\\
*&*&*&*\\
4x ^2+2xy +y ^2-1 & 4x ^2+2xy &-2xy  &
2xy+  y ^2
\end{array}
\right)
.
\eea
Write
\bea
\label{eq:wxyDef}
w_{x,y}
&=&\xi_{x,y}^{t}\cdot e_{4}
\\
\nonumber
&=&
(4x^{2}+2xy+y^{2}-1,4x^{2}+2xy,-2xy,2xy+y^{2})^{t}.
\eea
Then again by \eqref{eq:setB}, we have shown the following \pagebreak
\begin{lem}[\cite{SarnakToLagarias}]
Let $x,y\in\Z$ with $(2x,y)=1$, 
and take any element $\g\in\G$ with corresponding quadruple 
\be\label{eq:vgIs}
v_{\g}=(a_{\g},b_{\g},c_{\g},d_{\g})^{t}=\g\cdot v_{0}\in\sO.
\ee
Then the number
\be
\label{eq:binQuad}
\<e_{4},\xi_{x,y}\cdot \g\cdot v_{0}\>
=
\<w_{x,y},\g\cdot v_{0}\>
=
4A_{\g}x^{2}
+4B_{\g}xy
+C_{\g}y^{2} 
 -a _{\g}
\ee
is the curvature of some circle in $\sG$,
where we have defined 
\bea
\label{eq:ABCdef}
A_{\g}
&:=&
a_{\g}+b_{\g},
\\
\nonumber
B_{\g}
&:=&
 {a _{\g}
+ b  _{\g}
  - c _{\g}
+ d _{\g}\over 2},
\\
\nonumber
C_{\g}
&:=&
 a_{\g}
 + d_{\g}
 .
\eea
Note from \eqref{eq:Fv} that $B_{\g}$ is integral.
\end{lem}

Observe that, by construction, the value of $a_{\g}$ is unchanged under the orbit of 
the group
\eqref{eq:G1}, and 
the circles whose 
curvatures are generated by \eqref{eq:binQuad}  are all tangent to the circle corresponding to $a_{\g}$. 
It is classical (see \cite{Bernays1912})
that the number of distinct primitive values 
up to $N$
 assumed by a 
 positive-definite binary quadratic form  is of order at least $ N(\log N)^{-1/2}$, proving \eqref{eq:SarBnd}.

To fix notation, we define the binary quadratic appearing in \eqref{eq:binQuad} and its shift by
\be\label{eq:ffVdef}
f_{\g}(x,y)
:=
A_{\g}x^{2}+2B_{\g} xy+C_{\g}y^{2}
,
\qquad
\ff_{\g}(x,y):=
f_{\g}(x,y)
-a_{\g}
,
\ee
so that
\be\label{eq:ffToInProd}
\<
w_{x,y},
\g\cdot v_{0}\>
=
\ff_{\g}(2x,y)
.
\ee
Note
from \eqref{eq:ABCdef} and \eqref{eq:Fv}
 that the discriminant of $f_{\g}$ is
\be
\label{eq:disc}
\gD_{\g}
=
4(B_{\g}^{2}-A_{\g}C_{\g})
=
-4 a_{\g}^2
.
\ee

When convenient, we will drop the subscripts $\g$ in all the above.

\subsection{Congruence Subgroups}\

For each $q\ge1$, define the ``principal'' $q$-congruence subgroup
\be\label{eq:GqDef}
\G(q)
:=
\{
\g\in\G:
\g\equiv I(\mod q)
\}
.
\ee
These groups all have infinite index in $G(\Z)$, but finite index in $\G$. The quotients $\G/\G(q)$ have been determined 
completely
by Fuchs \cite{FuchsThesis} 
by proving an explicit
Strong Approximation theorem (see \cite{MatthewsVasersteinWeisfeiler1984}), Goursat's Lemma, and other ingredients, as we explain below. 
Since
$G$ does not itself have
the
 Strong Approximation
Property, 
we pass to its connected spin double cover $\SL_{2}(\C)$. We will need the covering map explicitly later, so record it here. 

First change variables from the Descartes form $F$ to 
$$
\tilde F(x,y,z,w):=xw+y^{2}+z^{2}.
$$
Then there is a homomorphism 
$\iota_{0}:\SL(2,\C)\to \SO_{\tilde F}(\R)$, sending 
$$
g=
\mattwo{a+\ga i}{b+\gb i}{c+\g i}{d+\gd i} \in\SL(2,\C)
$$
to
$$
{1\over |\det(g)|^{2}}
\left(
\begin{array}{cccc}
 a^2+\alpha ^2 & 2 (a c+\alpha  \gamma ) & 2 (c \alpha - a \gamma)  & -c^2-\gamma ^2 \\
 a b+\alpha  \beta  & b c+a d+\beta  \gamma +\alpha  \delta  & d \alpha +c \beta -b \gamma -a \delta  & -c d-\gamma  \delta 
   \\
 a \beta -b \alpha  & -d \alpha +c \beta -b \gamma +a \delta  & -b c+a d-\beta  \gamma +\alpha  \delta  & d \gamma -c \delta
    \\
 -b^2-\beta ^2 & -2 (b d+\beta  \delta ) & 2( b \delta -d \beta ) & d^2+\delta ^2
\end{array}
\right)
.
$$
To map from $\SO_{\tilde F}$ to $\SO_{F}$, we apply a
conjugation, see \cite[(4.1)]{GrahamLagarias2003}. Let 
\be\label{eq:iota}
\iota:\SL(2,\C)\to\SO_{F}(\R)
\ee 
be
the composition of this conjugation with $\iota_{0}$. Let $\tilde\G$ be the preimage of $\G$ under $\iota$.
\begin{lem}[\cite{GrahamLagariasMallowsWilksYanI, FuchsThesis}]
The group $\tilde\G$ is generated by 
$$
\pm
\mattwo1{4i}{}1,\quad
\pm
\mattwo{-2}ii{},\quad
\pm
\mattwo{2+2i}{4+3i}{-i}{-2i}.
$$
\end{lem}

With this explicit realization of $\tilde\G$ (and hence $\G$), Fuchs was able to 
  explicitly
 determine
   the images of $\tilde\G$ in $\SL(2,\Z[i]/(q))$, and hence understand the quotients $\G/\G(q)$ for all $q$.
\begin{lem}[\cite{FuchsThesis}]\label{lem:GGq}\

(1) The quotient groups $\G/\G(q)$ are multiplicative, that is, if $q$ factors as 
$$
q=p_{1}^{\ell_{1}}\cdots p_{r}^{\ell_{r}},
$$
then
$$
\G/\G(q)\cong \G/\G(p_{1}^{\ell_{1}})\times\cdots\times \G/\G(p_{r}^{\ell_{r}}).
$$

(2) If $(q,6)=1$ then 
\be\label{eq:Fuc}
\G/\G(q)\cong\SO_{F}(\Z/q\Z).
\ee

(3)
If $q=2^{\ell}$, $\ell\ge3$, then $\G/\G(q)$ is the full preimage of $\G/\G(8)$ under the projection $\SO_{F}(\Z/q\Z)\to \SO_{F}(\Z/8\Z)$. That is, the powers of $2$ stabilize at $8$. 
Similarly, the powers of $3$ stabilize at $3$, meaning that for $q=3^{\ell}$, $\ell\ge1$, the quotient $\G/\G(q)$ is the preimage of $\G/\G(3)$ under the corresponding projection map.
\end{lem}

\begin{rmk}\label{rmk:locA}
This of course explains all local obstructions, cf. \eqref{eq:locA}. 
The admissible numbers are precisely those
 residue classes $(\mod 24)$ which appear  as some entry in the orbit of $v_{0}$ under $\G/\G(24)$.
\end{rmk}

\newpage


\section{Setup and Outline of the Proof}\label{sec:outline}

In this section, we introduce the main exponential sum and give an outline of the rest of the argument. 
Recall  the fixed 
gasket
 $\sG$ having curvatures $\sB$ and root quadruple $v_{0}$.
Let $\G$ be  the Apollonian subgroup with subgroup $\Xi$, see \eqref{eq:G1}. Let $\gd\approx1.3$ be the Hausdorff dimension of the gasket $\sG$; see \S\ref{sec:preII} for the important role played by this geometric invariant.
Recall also from \eqref{eq:binQuad} that for any $\g\in\G$ and $\xi\in\Xi$, 
$$
\<e_{4},\xi\g v_{0}\>\in\sB.
$$
Our approach, mimicing \cite{BourgainKontorovich2010, BourgainKontorovich2011a}, is to exploit the bilinear (or multilinear) structure above.

We first give an informal description of the main ensemble from which we will form an exponential sum.
Let $N$ be our main growing parameter. 
We construct our ensemble by decomposing a ball in $\G$ of norm $N$ into two balls, 
a small
one in all of $\G$ of norm $T$, and a larger one of norm $X^{2}$ in $\Xi$, corresponding to $x,y\asymp X$. Specifically, we take
\be\label{eq:TXNis}
T=N^{1/100}\quad\text{and}\quad X=N^{99/200},\qquad\text{so that}\qquad TX^{2}=N.
\ee
See \eqref{eq:gS2p} and \eqref{eq:gS12p} where these numbers are used.

We further need the technical condition that in the $T$-ball, the value of $a_{\g}=\<e_{1},\g\, v_{0}\>$ (see \eqref{eq:vgIs}) is of order $T$. This is used crucially in 
\eqref{eq:cJfbnd}
and 
\eqref{eq:ACisT}. 

Finally, for technical reasons (see Lemma \ref{lem:spec3} below), we need to further split the $T$-ball into two: a small ball of norm $T_{1}$, and a big ball of norm $T_{2}$. Write
\be\label{eq:TT1T2}
T=T_{1}T_{2},\qquad T_{2}=T_{1}^{\cC},
\ee
where $\cC$ is a large constant depending only on the spectral gap for $\G$; it is determined in \eqref{eq:cCis}. We now make formal the above discussion.

\subsection{Introducing the Main Exponential Sum}\

 Let $N,X, T, T_{1}$, and $T_{2}$ be as in \eqref{eq:TXNis} and \eqref{eq:TT1T2}. Define the family
\be\label{eq:fFdef}
\fF=\fF_{T}:=
\left\{
\g=\g_{1}\g_{2}:
\begin{array}{c}
\g_{1},\g_{2}\in\G,\\
T_{1}<\|\g_{1}\|<2T_{1},\\
T_{2}<\|\g_{2}\|<2T_{2},\\
\<e_{1},\g_{1}\,\g_{2}\,v_{0}\>>T/100
\end{array}
\right\}
.
\ee
From Lax-Phillips \cite{LaxPhillips1982} 
(or see \eqref{eq:Vin}),
we have the bound
\be\label{eq:fFTbnd}
\#\fF_{T}\ll T^{\gd}.
\ee

From \eqref{eq:ffToInProd}, we can identify $\g\in\fF$ with a shifted binary quadratic form $\ff_{\g}$ of discriminant $-4a_{\g}^{2}$ via
$$
\ff_{\g}(2x,y)=\<w_{x,y},\g\, v_{0}\>.
$$
Recall from \eqref{eq:binQuad} that whenever $(2x,y)=1$, the above is a curvature in the 
gasket. We sometimes drop $\g$, writing simply $\ff\in\fF$; then the latter can also be thought of as a family of shifted quadratic forms. Note also that the decomposition $\g=\g_{1}\g_{2}$ in \eqref{eq:fFdef} need not be unique, so some forms may appear with multiplicity.

One final technicality is to smoothe the sum on $x,y\asymp X$. To this end, we fix a smooth, nonnegative function 
 $\gU$, supported in $[1,2]$ and having unit mass,
$
\int_{\R}\gU(x)dx=1.
$
\\

Our main object of study is then the representation number
\be\label{eq:cRNis}
\cR_{N}(n):=\sum_{\ff\in\fF_{T}}\sum_{(2x,y)=1}
\gU\left(\frac {2x}X\right)
\gU\left(\frac yX\right)
\bo_{\{n=\ff(2x,y)\}}
,
\ee
and the corresponding exponential sum, its Fourier transform
\be\label{eq:cRNhatIs}
\widehat{\cR_{N}}(\gt):=\sum_{\ff\in\fF}\sum_{(2x,y)=1}
\gU\left(\frac {2x}X\right)
\gU\left(\frac yX\right)
e(\gt\, \ff(2x,y))
.
\ee
Clearly $\cR_{N}(n)\neq0$ implies that $n\in\sB$. Note also from \eqref{eq:fFTbnd} that the total mass satisfies
\be\label{eq:totMass}
\widehat{\cR_{N}}(0)\ll T^{\gd}X^{2}.
\ee

The condition $(2x,y)=1$ will be a technical nuisance,
and can be freed by a standard use of the M\"obius inversion formula. To this end,
  we introduce another parameter 
\be\label{eq:Uis}
 U=N^{\fu},
\ee
a small power of $N$, with $\fu>0$  depending only on the spectral gap of $\G$; it is determined in \eqref{eq:fuIs}. Then by truncating M\"obius inversion, define
\be\label{eq:cRNhatUis}
\widehat{\cR_{N}^{U}}(\gt)
:=
\sum_{\ff\in\fF}\sum_{x,y\in\Z}
\gU\left(\frac {2x}X\right)
\gU\left(\frac yX\right)
e(\gt\, \ff(2x,y))
\sum_{u\mid(2x,y)\atop u<U}\mu(u)
,
\ee
with corresponding ``representation 
function'' $\cR_{N}^{U}$ (which could be negative).
\\

\subsection{Reduction to the Circle Method}\

We are now in position to outline the argument in the rest of the paper. Recall that $\sA$ is the set  of admissible numbers. We first reduce our main Theorem \ref{thm:Main} to the following
\begin{thm}\label{thm:RNU} 
There exists an $\eta>0$ and a function $\fS(n)$ with the following properties. For $\foh N<n<N$, the singular series $\fS(n)$ is nonnegative, vanishes only when $n\notin\sA$, and is otherwise 
$
\gg_{\vep}N^{-\vep}
$ 
for any $\vep>0$. Moreover, for $\foh N<n<N$ and admissible,
\be\label{eq:RNUbnd}
\cR_{N}^{U}(n)\gg \fS(n) T^{\gd-1},
\ee
except for a set of cardinality $\ll N^{1-\eta}$.
\end{thm}

\pf[Proof of Theorem \ref{thm:Main} assuming Theorem \ref{thm:RNU}:]\

We first show that the difference between $\cR_{N}$ and $\cR_{N}^{U}$ is small in $\ell^{1}$. Using \eqref{eq:fFTbnd}
we have
\beann
\sum_{n<N}|\cR_{N}(n)-\cR_{N}^{U}(n)|
&=&
\sum_{n<N}
\left|
\sum_{\ff\in\fF}\sum_{x,y\in\Z}
\gU\left(\frac {2x}X\right)
\gU\left(\frac yX\right)
\bo_{\{n=\ff(2x,y)\}}
\sum_{u\mid(2x,y)\atop u\ge U}\mu(u)
\right|
\\
&\ll&
\sum_{\ff\in\fF}
\sum_{u
\ge U}
\sum_{y\ll X\atop y\equiv0(\mod u)}
\sum_{x\ll X\atop 2x\equiv0(\mod u)}
1
\\
&\ll
&
T^{\gd}
{X^{2}\over U}
,
\eeann
for any $\vep>0$. Recall from \eqref{eq:Uis} that $U$ is a fixed power of $N$, so the above saves a power from the total mass \eqref{eq:totMass}.

Now let $Z$ be the ``exceptional'' set of admissible $n<N$ for which $\cR_{N}(n)=0$. Futhermore, let $W$ be the set of admissible $n<N$ for which \eqref{eq:RNUbnd} is satisfied. Then
\beann
T^{\gd}
{X^{2}\over U}
&\gg
&
\sum_{n<N} |\cR_{N}^{U}(n)-\cR_{N}(n)|
\ge
\sum_{n\in Z\cap W} |\cR_{N}^{U}(n)-\cR_{N}(n)|
\\
&\gg_{\vep}&
| Z\cap W|
\cdot
 T^{\gd-1}
N^{-\vep}
 .
\eeann
Note also from Theorem \ref{thm:RNU} that $|Z\cap W^{c}|
\le |W^{c}|\ll N^{1-\eta}$.
Hence by \eqref{eq:TXNis} and \eqref{eq:Uis},
\be\label{eq:nZbnd}
|Z|
=
| Z\cap W^{c}|
+
| Z\cap W|
\ll_{\vep}
N^{1-\eta}
+
{N^{1+\vep}\over U}
,
\ee
which
 is a power savings since $\vep>0$ is arbitrary. This completes the proof.
\epf

To establish \eqref{eq:RNUbnd}, we decompose $\cR_{N}^{U}$ into ``major'' and ``minor'' arcs, reducing Theorem \ref{thm:RNU} to the following
\begin{thm}\label{thm:CircMeth}
There exists an $\eta>0$ and a decomposition 
\be\label{eq:RNcMcE}
\cR_{N}^{U}(n)=\cM_{N}^{U}(n)+\cE_{N}^{U}(n)
\ee
with the following properties. For $\foh N<n<N$ and admissible, $n\in\sA$, we have
\be\label{eq:cMNUbnd}
\cM_{N}^{U}(n)\gg \fS(n) T^{\gd-1},
\ee
except for a set of cardinality $\ll N^{1-\eta}$. The singular series $\fS(n)$ is the same as in Theorem \ref{thm:RNU}. Moreover,
\be\label{eq:cEN2bnd}
\sum_{n<N}|\cE_{N}^{U}(n)|^{2}
\ll 
N\, T^{2(\gd-1)} N^{-\eta}
.
\ee
\end{thm}
\pf[Proof of Theorem \ref{thm:RNU} assuming Theorem \ref{thm:CircMeth}:]\

We restrict our attention to the set of admissible $n<N$ so that \eqref{eq:cMNUbnd} holds (the remainder having sufficiently small cardinality). Let $Z$ denote the subset of these $n$ for which $\cR_{N}^{U}(n)<\foh \cM_{N}^{U}(n)$; hence for $n\in Z$, 
$$
1\ll
{|\cE_{N}^{U}(n)|\over N^{-\vep}T^{\gd-1}}
.
$$ 
Then by \eqref{eq:cEN2bnd},
\beann
|Z|
&\ll_{\vep}&
\sum
_{n<N}
{|\cE_{N}^{U}(n)|^{2}
\over 
N^{-\vep}T^{2(\gd-1)}}
\ll
N^{1-\eta+\vep}
,
\eeann 
whence the claim follows, since $\vep>0$ is arbitrary.
\epf

\subsection{Decomposition into Major and Minor Arcs}\

Next we 
explain
the decomposition \eqref{eq:RNcMcE}. Let $M$ be a parameter controlling  the depth of approximation in Dirichlet's theorem:
 for any irrational $\gt\in[0,1]$, there exists some $q<M$ and $(r,q)=1$ so that $|\gt-r/q|<1/(qM)$. We will eventually set 
\be\label{eq:Mis}
M=XT,
\ee
see \eqref{eq:nmRest} 
where this value is used. (Note that $M$ is a bit bigger than $N^{1/2}=XT^{1/2}$.) 

Writing $\gt=r/q+\gb$, we introduce parameters 
\be\label{eq:Q0K0intro}
Q_{0}, K_{0},
\ee 
small powers of $N$ as determined in  \eqref{eq:Q0is}
, so that the ``major arcs'' correspond to $q<Q_{0}$ and $|\gb|<K_{0}/N$.
In fact, we need a smooth version of this decomposition. 

To this end, recall the ``hat'' function and its Fourier transform
\be\label{eq:hatFunc}
\ft(x):=\min(1+x,1-x)^{+},\qquad
\hat\ft(y)=\left({\sin(\pi y)\over \pi y}\right)^{2}.
\ee
Localize $\ft$ to the width $K_{0}/N$, periodize it to the circle, and put this spike on each fraction in the major arcs:
\be\label{eq:hatFuncN}
\fT(\gt)
=
\fT_{N,Q_{0},K_{0}}(\gt):=\sum_{q<Q_{0}}\sum_{(r,q)=1}\sum_{m\in\Z}\ft\left({N\over K_{0}}\left(\gt+m-\frac rq\right)\right)
.
\ee
By construction, $\fT$ lives on the circle $\R/\Z$ and is supported within $K_{0}/N$ of fractions $r/q$ with small denominator, $q<Q_{0}$, as desired.

Then define the ``main term''
\be\label{eq:cMNUdef}
\cM_{N}^{U}(n):=
\int_{0}^{1}
\fT(\gt)
\widehat{
\cR_{N}^{U}}
(\gt)
e(-n\gt)
d\gt
,
\ee
and ``error term''
\be\label{eq:cENUdef}
\cE_{N}^{U}(n):=
\int_{0}^{1}
(1-\fT(\gt))
\widehat{
\cR_{N}^{U}}
(\gt)
e(-n\gt)
d\gt
,
\ee
so that \eqref{eq:RNcMcE} obviously holds.

Since $\cR_{N}^{U}$ could be negative, the same holds for $\cM_{N}^{U}$. Hence we will establish \eqref{eq:cMNUbnd} by first proving a related result for 
\be\label{eq:cMNdef}
\cM_{N}(n):=
\int_{0}^{1}
\fT(\gt)
\widehat{
\cR_{N}}
(\gt)
e(-n\gt)
d\gt
,
\ee
and then showing that $\cM_{N}$ and $\cM_{N}^{U}$ cannot differ by too much for too many values of $n$. This is the same (but in reverse) as the transfer from $\cR_{N}$ to $\cR_{N}^{U}$ in \eqref{eq:nZbnd}. See Theorem \ref{thm:cMNis} for the lower bound on $\cM_{N}$, and Theorem \ref{thm:cMNtoMNU} for the transfer.

To prove \eqref{eq:cEN2bnd}, we apply Parseval and decompose dyadically:
\beann
\sum_{n}|\cE_{N}^{U}(n)|^{2}
&=&
\int_{0}^{1}
|1-\fT(\gt)|^{2}
\left|
\widehat{\cR_{N}^{U}}(\gt)
\right|^{2}
d\gt
\\
&\ll&
\cI_{Q_{0},K_{0}}
+
\cI_{Q_{0}}
+
\sum_{Q_{0}\le Q<M\atop \text{dyadic}}
\cI_{Q}
,
\eeann
where we have dissected the circle into the following regions (using that $|1-\ft(x)|=|x|$ on $[-1,1]$):
\bea
\label{eq:IQ0K0}
\cI_{Q_{0},K_{0}}
&:=&
\int\limits_{\gt=\frac rq+\gb \atop q<Q_{0},(r,q)=1,|\gb|<K_{0}/N}
\left|\gb\frac N{K_{0}}\right|^{2}
\left|
\widehat{\cR_{N}^{U}}(\gt)
\right|^{2}
d\gt
,
\\
\label{eq:IQ0}
\cI_{Q_{0}}
&:=&
\int\limits_{\gt=\frac rq+\gb \atop q<Q_{0},(r,q)=1,K_{0}/N<|\gb|<1/(qM)}
\left|
\widehat{\cR_{N}^{U}}(\gt)
\right|^{2}
d\gt
,
\\
\label{eq:IQdef}
\cI_{Q}
&:=&
\int\limits_{\gt=\frac rq+\gb \atop Q\le q<2Q,(r,q)=1,|\gb|<1/(qM)}
\left|
\widehat{\cR_{N}^{U}}(\gt)
\right|^{2}
d\gt
.
\eea

Bounds of the quality \eqref{eq:cEN2bnd}  are given for \eqref{eq:IQ0K0} and \eqref{eq:IQ0} in \S\ref{sec:qQ0}, see Theorem \ref{thm:IQ0}. Our estimation of \eqref{eq:IQdef} decomposes further into two cases, whether $Q<X$ or $X\le Q<M$, and are handled separately in \S\ref{sec:QX} and \S\ref{sec:QXT}; see Theorems \ref{thm:IQX} and \ref{thm:IQM}, respectively.

We point out again that our averaging on $n$ in the minor arcs makes this quite crude as far as individual $n$'s (the subject of Conjecture \ref{conj}) are concerned. 

\subsection{The Rest of the Paper}\

The only section not yet described is \S\ref{sec:lems}, where we furnish some  lemmata  which are useful in the sequel. These decompose into two categories: one set of lemmata is related to
some infinite-volume counting problems, for which the background in \S\ref{sec:preII} is indispensable. 
The other lemma is of a classical flavor, corresponding to a local analysis for the shifted binary form $\ff$; this studies a certain exponential sum which is dealt with via Gauss and Kloosterman/Sali\'e sums.
\\

This completes our outline of the rest of the paper.


\newpage

\section{Preliminaries II: Automorphic Forms and Representations}\label{sec:preII}


\begin{figure}
\begin{center}

\vskip-.3in
\includegraphics[width=3in]{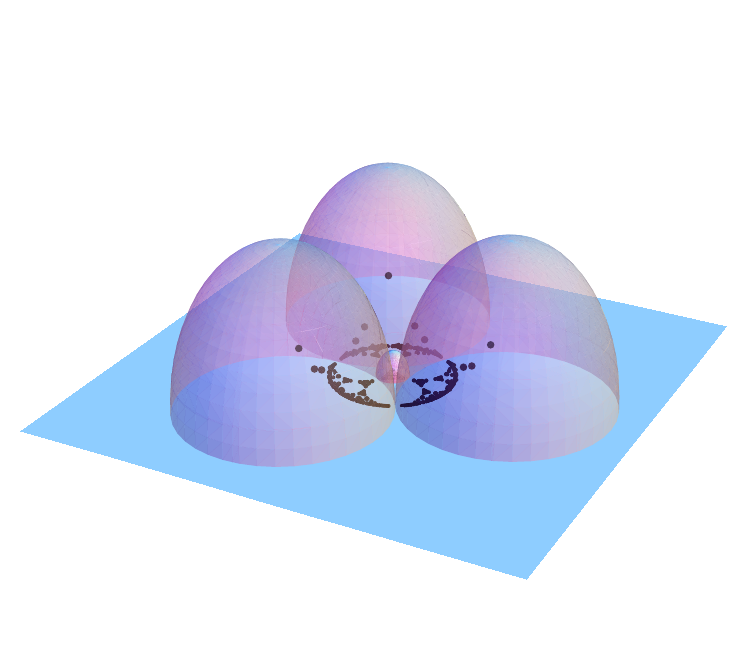}
\end{center}
\vskip-.3in
\caption{The orbit of a point in hyperbolic space under the Apollonian group.}
\label{fig:lim}
\end{figure}

\subsection{Spectral Theory}\

Recall the general spectral theory in our present context. 
We abuse notation (in this 
section only), passing from $G=\SO_{F}(\R)$ to its spin double cover $G=\SL(2,\C)$. Let $\G<G$ be a geometrically finite discrete group. (The Apollonian group is such, being a Schottky group, see Fig. \ref{fig:lim}.) Then $\G$ acts discontinuously on the upper half space $\bH^{3}$, and any $\G$ orbit has a limit set $\gL_{\G}$ in the boundary $\dd\bH^{3}\cong
S^{2}
$ of some Hausdorff dimension $\gd=\gd(\G)\in[0,2]$. We assume that $\G$ is non-elementary (not virtually abelian), so $\gd>0$, and moreover that $\G$ is not a lattice,
that is, the quotient $\G\bk\bH^{3}$ has infinite hyperbolic volume; 
then $\gd<2$. 
 The hyperbolic Laplacian $\gD$ acts on the space $L^{2}(\G\bk \bH^{3})$ of functions automorphic under $\G$ and square integrable on the quotient; we choose the Laplacian to be positive definite. The spectrum is controlled via the following, see \cite{Patterson1976, Sullivan1984, LaxPhillips1982}.
\begin{thm}[Patterson, Sullivan, Lax-Phillips]
The spectrum above $1$ is purely continuous, and the spectrum below $1$ is purely discrete. The latter is empty unless $\gd>1$, in which case, ordering the eigenvalues by
\be\label{eq:gls}
0<\gl_{0}<\gl_{1}
\le
\cdots\le
\gl_{max}<1,
\ee
the base eigenvalue $\gl_{0}$ is given by
$$
\gl_{0}=\gd(2-\gd).
$$
\end{thm}

\begin{rmk}
In our application to the Apollonian group, the limit set is precisely the underlying gasket, see Fig. \ref{fig:lim}. It has dimension
\be\label{eq:dim}
\gd\approx 1.3...>1.
\ee
\end{rmk}

Corresponding to $\gl_{0}$ is the Patterson-Sullivan base eigenfunction, 
 $\varphi_{0}$, which can be realized explicitly as the integral of a Poisson kernel against the so-called Patterson-Sullivan measure $\mu$.
Roughly speaking, $\mu$ is the weak$^{*}$ limit as $s\to\gd^{+}$ of the measures
\be\label{eq:muPS}
\mu_{s}(x):=
{
\sum_{\g\in\G}\exp({-s\,d(\fo,\g\cdot\fo)){\bf 1}_{x=\g\fo}}
\over
\sum_{\g\in\G}\exp({-s\,d(\fo,\g\cdot\fo))}
}
,
\ee
where $d(\cdot,\cdot)$ is the hyperbolic distance, and $\fo$ is any fixed point in $\bH^{3}$.
\\

\subsection{Spectral Gap}\label{sec:specGap}\

We assume henceforth that $\G$ 
moreover 
satsifies
 $\G<\SL(2,
\cO
)$, where $\cO=\Z[i]$. Then we have a tower of congruence subgroups: for any integer $q\ge1$, define $\G(q)$ to be the kernel of the projection map $\G\to \SL(2,
\cO
/\fq)$, with $\fq=(q)$ the principal ideal. As in \eqref{eq:gls}, write
\be\label{eq:glQs}
0<\gl_{0}(q)<\gl_{1}(q)
\le
\cdots\le
\gl_{max(q)}(q)<1,
\ee
for the discrete spectrum of $\G(q)\bk\bH^{3}$. The groups $\G(q)$, while of infinite covolume, have finite index in $\G$, and hence 
\be\label{eq:gl0Qgl0}
\gl_{0}(q)=\gl_{0}=\gd(2-\gd).
\ee 
But the second eigenvalues $\gl_{1}(q)$ could {\it a priori} encroach on the base. The fact that this does not happen is the spectral gap property for $\G$.
\begin{thm}
\label{thm:specGap}
Given $\G$ as above, there exists some $\vep=\vep(\G)>0$ such that for all $q\ge1$,
\be\label{eq:specGap}
\gl_{1}(q)\ge \gl_{0}+\vep.
\ee
\end{thm}
This is proved in the Appendix by P\'eter Varj\'u.

\subsection{Representation Theory and Mixing Rates}\

By the Duality Theorem of Gelfand, Graev, and Piatetski-Shapiro \cite{GelfandGraevPS1966}, the spectral decomposition above is equivalent to the decomposition into irreducibles of the right regular representation 
acting on $L^{2}(\G\bk G)$. That is, we identify $\bH^{3}\cong G/K$, with $K=\SU(2)$ a maximal compact subgroup, and lift functions from $\bH^{3}$ to (right $K$-invariant) functions on $G$. Corresponding to \eqref{eq:gls} is the decomposition
\be\label{eq:RepDecomp}
L^{2}(\G\bk G)=
V_{\gl_{0}}\oplus
V_{\gl_{1}}\oplus
\cdots
\oplus
V_{\gl_{max}}\oplus
V_{temp}
.
\ee
Here $V_{temp}$ contains the tempered spectrum (for $\SL_{2}(\C)$, every non-spherical irreducible representation is tempered), and each $V_{\gl_{j}}$ is an infinite dimensional vector space, isomorphic as a $G$-representation to a complementary series representation with parameter $s_{j}\in(1,2)$ determined by $\gl_{j}=s_{j}(2-s_{j})$.
Obviously, a similar decomposition holds for $L^{2}(\G(q)\bk G)$, corresponding to \eqref{eq:glQs}.

We also have the following well-known general fact about mixing rates of matrix coefficients, see e.g. \cite{CowlingHaagerupHowe1988}. 
First we recall the relevant Sobolev norm. Let $(\pi,V)$ be a unitary $G$-representation, and 
let $\{X_{j}\}$ denote an orthonormal basis of the Lie algebra $\fk$ of $K$ with respect to an $Ad$-invariant scalar product. For a smooth vector $v\in V^{\infty}$, define 
the (second order) Sobolev norm $\cS$ of $v$ by
$$
\cS v
 := 
 \|v\|_{2}
+\sum_{j} \|d\pi(X_{j}).v\|_{2}
 +\sum_{j} \sum_{j'} \|d\pi(X_{j})d\pi(X_{j'}).v\|_{2}
.
$$ 
\begin{thm}[{\cite[Prop. 5.3]{KontorovichOh2011}}]
Let $\gT>1$ and $(\pi,V)$ be a unitary representation of $G$ which does not weakly contain any complementary series representation with parameter $s>\gT$. Then for any smooth vectors $v,w\in V^{\infty}$,
\be\label{eq:decayMtrx}
\left|
\<\pi(g).v,w\>
\right|
\ll
\|g\|^{-2(2-\gT)}
\cdot
\cS v\cdot
\cS w
.
\ee
Here $\|\cdot\|$ is the standard Frobenius matrix norm.
\\
\end{thm}

\subsection{Effective Bisector Counting}\

The next ingredient which we require is the recent work by Vinogradov \cite{Vinogradov2013} on effective bisector counting for such infinite volume quotients. Recall the following sub
(semi)groups
 of $G$: 
$$
A=\left\{a_{t}:=\mattwo{e^{t/2}}{}{}{e^{-t/2}}:t\in\R
\right\},
A^{+}=\left\{a_{t}:t\ge0
\right\},
$$
$$
M=\left\{\mattwo{e^{2\pi i\gt}}{}{}{e^{-2\pi i\gt}}:\gt\in\R/\Z
\right\},
K=\SU(2)
.
$$
We have the Cartan decomposition $G=KA^{+}K$, unique up to the normalizer $M$ of $A$ in $K$. We require it in the following more precise form. Identify $K/M$ with the sphere $S^{2}\cong\dd\bH^{3}.$ Then for every $g\in G$ not in $K$, there is a unique decomposition
\be\label{eq:gDecomp}
g=s_{1}(g)\cdot a(g)\cdot m(g) \cdot s_{2}(g)^{-1}.
\ee
with $s_{1},s_{2}\in K/M$, $a\in A^{+}$ and  $m\in M$, corresponding to
$$
G=K/M\times A^{+}\times M\times M\bk K,
$$
see, e.g., \cite[(3.4)]{Vinogradov2013}. The following theorem follows easily 
from \cite[Thm 2.2]{Vinogradov2013}.

\begin{thm}[{\cite{Vinogradov2013}}]\label{thm:Vin}
Let $\Phi,\Psi\subset S^{2}$ be 
spherical caps
 and let $\cI\subset\R/\Z$ be an interval. Then under the above hypotheses on $\G$ (in particular $\gd>1$), and using the decomposition 
\eqref{eq:gDecomp},
we have
\be\label{eq:Vin}
\sum_{\g\in\G}
\bo
{
\left\{
\begin{array}{c}
s_{1}(\g)\in\Phi
\\
s_{2}(\g)\in \Psi
\\
\|a(\g)\|^{2}<T
\\
m(\g)\in\cI
\end{array}
\right\}
}
=
c_{\gd}\cdot
\mu(\Phi)
\mu(\Psi)
\ell(\cI)
T^{\gd}
+
O\big(
T^{
\gT
}
\big)
,
\ee
as $T\to\infty$.
Here $c_{\gd}>0$, $\|\cdot\|$ is the  Frobenius norm, $\ell$ is Lebesgue measure, $\mu$ is Patterson-Sullivan measure (cf. \eqref{eq:muPS}),
  and 
\be\label{eq:ThIs}
\gT<\gd
\ee
depends only on the spectral gap for $\G$. The implied constant does not depend on $\Phi,\Psi,$ or $\cI$.
\end{thm}

This generalizes from $\SL(2,\R)$ to $\SL(2,\C)$ the main result of \cite{BourgainKontorovichSarnak2010}, which is itself a 
generalization (with weaker exponents) to
 our infinite volume setting
  of \cite[Thm 4]{GoodBook}.

\newpage


\section{Some Lemmata}\label{sec:lems}

\subsection{Infinite Volume Counting Statements}\label{sec:count}\

Equipped with the tools of 
\S\ref{sec:preII}, we isolate here some consequences 
which
will be needed in the sequel. 
%
We return to the notation $G=\SO_{F}$, with $F$ the Descartes
form
 \eqref{eq:Fdef}, 
$\G=\cA\cap G$, the orientation preserving Apollonian subgroup,
and $\G(q)$ its principal congruence subgroups.
Moreover, 
we import all the notation from the previous section.


First we use the spectral gap to see that summing over a coset of a congruence group can be reduced to summing over the original group.

\begin{lem}\label{lem:spec0}
Fix $\g_{1}\in\G$, $q\ge1$, and  any ``congruence'' group $\tilde\G(q)$ satisfying
\be\label{eq:G1G}
\G(q)<\tilde\G(q)<\G.
\ee
Then as $Y\to\infty$,
\bea
\label{eq:lemLHS}
&&
\hskip-.5in
\#\{
\g\in\tilde\G(q)
:
\|\g_{1}\g\|<Y
\}
\\
\label{eq:lemRHS}
&&
=
{1\over [\G:\tilde\G(q)]}
\cdot
\#\{
\g\in\G
:
\|\g\|<Y
\}
+
O(Y^{\gT_{0}})
,
\eea
where $\gT_{0}<\gd$ depends only on the spectral gap for $\G$. The implied constant above 
does not depend on $q$ or $\g_{1}$.  
The same holds with $\g_{1}\g$ in \eqref{eq:lemLHS} replaced by $\g\g_{1}$.
\end{lem}


This  simple lemma  follows from a more-or-less standard argument. We give a sketch below, since a slightly more complicated result will be needed later, cf. Lemma \ref{lem:spec1}, but with essentially no  new ideas. After proving the lemma below, we will use the argument as a template for the more complicated statement.

\pf[Sketch of Proof]\

Denote the left hand side \eqref{eq:lemLHS} by $\cN_{q}$, and let $\cN_{1}/[\G:\tilde\G(q)]$ be the first term of \eqref{eq:lemRHS}. For $g\in G$, let 
\be\label{eq:fgIs}
f(g)=f_{Y}(g):=\bo_{\{\|g\|<Y\}},
\ee
and define
\be\label{eq:FqIs}
F_{q}(g,h):=
\sum_{\g\in\tilde\G(q)}
f(g^{-1}\g h),
\ee
so that 
\be\label{eq:cNq}
\cN_{q}=F_{q}(\g_{1}^{-1},e).
\ee

By construction, $F_{q}$ is a function on $\tilde\G(q)\bk  G\times \tilde\G(q)\bk G$, and we smooth $F_{q}$ in
both
 copies of $\tilde\G(q)\bk G$, as follows. 
Let $\psi\ge0$ be a smooth bump function supported in a ball of radius $\eta>0$ (to be chosen later) about the origin in $G$ with $\int_{G}\psi=1$, and 
automorphize it to
$$
\Psi_{q}(g):=\sum_{\g\in\tilde\G(q)}\psi(\g g)
.
$$
Then clearly $\Psi_{q}$ is a bump function in $\tilde\G(q)\bk G$ with $\int_{\tilde\G(q)\bk G}\Psi_{q}=1$.
Let
$$
\Psi_{q,\g_{1}}(g):=\Psi_{q}(g\g_{1})
.
$$

Smooth the variables $g$ and $h$ in $F_{q}$ by 
considering
\beann
\cH_{q}
&:=&
\<F_{q},\Psi_{q,\g_{1}}\otimes\Psi_{q}\>
=
\int_{\tilde\G(q)\bk G}
\int_{\tilde\G(q)\bk G}
F_{q}(g,h)\Psi_{q,\g_{1}}(g)\Psi_{q}(h) dg\, dh
\\
&=&
\sum_{\g\in\tilde\G(q)}
\int_{\tilde\G(q)\bk G}
\int_{\tilde\G(q)\bk G}
f(\g_{1}g^{-1}\g h)\Psi_{q}(g)\Psi_{q}(h) dg\, dh
.
\eeann

First we estimate the error from smoothing:
\beann
\cE
&=&
|\cN_{q}-\cH_{q}|
\\
&\le&
\sum_{\g\in\G}
\int_{\tilde\G(q)\bk G}
\int_{\tilde\G(q)\bk G}
|f(\g_{1}g^{-1}\g h)
-f(\g_{1}\g)
|
\Psi_{q}(g)\Psi_{q}(h) dg\, dh
,
\eeann
where we have increased $\g$ to run over all of $\G$. The analysis splits into three ranges.
\begin{enumerate}
\item
If $\g$ is such that 
\be\label{eq:reg1Lem}
\|\g_{1}\g
\|>Y(1+10\eta),
\ee 
 then both $f(\g_{1}g^{-1}\g h)$ and $f(\g_{1}\g)$ vanish.
\item
In the range
\be\label{eq:reg2Lem}
\|\g_{1}\g\|<Y(1-10\eta),
\ee
both $f(\g_{1}g^{-1}\g h)$ and $f(\g_{1}\g)$ are $1$, so their difference vanishes.
\item In the intermediate range, we
apply \cite{LaxPhillips1982}, bounding the count by
\be\label{eq:reg3Lem}
\ll Y^{\gd}\eta + Y^{\gd-\vep},
\ee
where $\vep>0$ depends on the spectral gap for $\G$.
\end{enumerate}
Thus it remains to analyze $\cH_{q}$.

Use a simple change of variables (see \cite[Lemma 3.7]{BourgainKontorovichSarnak2010}) to express $\cH_{q}$ via matrix coefficients:
$$
\cH_{q}
=
\int_{G}
f(g)
\<
\pi(g)
\Psi_{q},
\Psi_{q,\g_{1}}
\>_{\tilde\G(q)\bk G}
dg.
$$
Decompose the matrix coefficient into its projection onto the base irreducible $V_{\gl_{0}}$ in
\eqref{eq:RepDecomp} and an orthogonal term, and bound the remainder by the mixing rate \eqref{eq:decayMtrx} using the uniform spectral gap $\vep>0$ in \eqref{eq:specGap}. The functions $\psi$ are bump functions in six real dimensions, so can be chosen 
to have second-order Sobolev norms bounded by $\ll \eta^{-5}$. 
Of course the projection onto the base representation is just $[\G:\tilde\G(q)]^{-1}$ times the same projection at level one, cf. \eqref{eq:gl0Qgl0}.
Running the above argument in reverse at level one (see \cite[Prop. 4.18]{BourgainKontorovichSarnak2010}) gives:
\be\label{eq:NqErr}
\cN_{q}=
{1\over[\G:\tilde\G(q)]}\cdot \cN_{1}
+
O(\eta  Y^{\gd}+Y^{\gd-\vep})
+
O(Y^{\gd-\vep} \eta^{-10})
.
\ee
Optimizing $\eta$ and renaming $\gT_{0}<\gd$ in terms of the spectral gap $\vep$ gives 
 the claim. 
 \epf

Next we exploit the previous lemma and the product structure of the family $\fF$
in \eqref{eq:fFdef}
 to save a small power of $q$ in the following modular restriction. Such a bound is needed at several places in \S\ref{sec:QX}.

\begin{lem}\label{lem:spec3}
Let $\gT_{0}$ be as in \eqref{eq:lemRHS}. Define $\cC$ in \eqref{eq:TT1T2} by
\be\label{eq:cCis}
\cC:={10^{30}\over  \gd-\gT_{0}},
\ee
hence determining $T_{1}$ and $T_{2}$.
  There exists some $\eta_{0}>0$ depending only on the spectral gap of $\G$ so that
for any $1\le q<N$ and any $r(\mod q)$,
\be\label{eq:spec3}
\sum_
{\g\in\fF}
\bo_{\{
\<e_{1},\g v_{0}\>
\equiv r(\mod q)\}}
\ll
{1\over q^{\eta_{0}}}
T^{\gd}
.
\ee
The implied constant is independent of $r$.
\end{lem}
\pf
Dropping the condition $\<e_{1},\g_{1}\,\g_{2}v_{0}\>>T/100$ in \eqref{eq:fFdef}, bound the left hand side of \eqref{eq:spec3} by
\be\label{eq:1}
\sum_
{\g_{1}\in\G\atop\|\g_{1}\|\asymp T_{1}}
\sum_
{\g_{2}\in\G\atop\|\g_{2}\|\asymp T_{2}}
\bo_{\{
\<e_{1},\g_{1}\g_{2} v_{0}\>
\equiv r(\mod q)\}}
\ee
We decompose the argument into two ranges of $q$. 

{\bf Case 1: $q$ small.}
In this range, we fix $\g_{1}$, and follow a standard argument for $\g_{2}$. 
Let $\tilde\G(q)<\G$ denote the stabilizer of $v_{0} (\mod q)$, that is 
\be\label{eq:G0q}
\tilde\G(q):=\{\g\in\G:\g v_{0}\equiv v_{0}(\mod q)\}.
\ee
Clearly \eqref{eq:G1G} is satisfied, and it is elementary  that 
\be\label{eq:G0qBnd}
[\G:\tilde\G(q)]\asymp q^{2},
\ee
cf.  \eqref{eq:Fuc}.
Decompose $\g_{2}=\g_{2}'\g_{2}''$ with $\g_{2}''\in\tilde\G(q)$ and $\g_{2}'\in \G/\tilde\G(q)$.
Then by \eqref{eq:lemRHS} and \cite{LaxPhillips1982}, we have
\beann
\eqref{eq:1}
&=&
\sum_
{\g_{1}\in\G\atop\|\g_{1}\|\asymp T_{1}}
\sum_{\g_{2}'\in\G/\tilde\G(q)}
\bo_{\{
\<e_{1},\g_{1}\g_{2}' v_{0}\>
\equiv r(\mod q)\}}
\sum_
{\g_{2}''\in\tilde\G(q)\atop\|\g_{2}'\g_{2}''\|\asymp T_{2}}
1
\\
&\ll&
T_{1}^{\gd}\
q\
\left(
\frac1{q^{2}}\
T_{2}^{\gd}
+
T_{2}^{\gT_{0}}
\right)
.
\eeann
Hence we have saved a whole power of $q$, as long as
\be\label{eq:q1range}
q<T_{2}^{(\gd-\gT_{0})/2}
.
\ee

{\bf Case 2: $q\ge T_{2}^{{\gd-\gT_{0}\over 2}}$.}
Then by  \eqref{eq:cCis} and \eqref{eq:TT1T2}, $q$ is actually a 
very large
power of $T_{1}$,  
\be\label{eq:qToT1}
q\ge T_{1}^{10^{29}}
.
\ee
In this range, 
we exploit Hilbert's Nullstellensatz and effective versions of Bezout's theorem; 
see a related argument in \cite[Proof of Prop. 4.1]{BourgainGamburd2009}.

 Fixing $\g_{2}$ in \eqref{eq:1} (with $\ll T_{2}^{\gd}$ choices), we set
$$
v:=\g_{2}v_{0},
$$
and play now with $\g_{1}$. Let $S$ be the set of $\g_{1}$'s in question (and we now drop the subscript $1$):
$$
S=S_{v,q}(T_{1}):=\{\g\in\G:\|\g\|\asymp T_{1}, \<e_{1},\g v\>\equiv r(\mod q)\}.
$$
This congruence restriction is to a modulus much bigger than the parameter, so we

{\bf Claim:} There is an integer vector $
v_{*}\neq0$ and an integer $z_{*}$ such that 
\be\label{eq:wStar}
\<e_{1}
,
\g v_{*}\>=z_{*}
\ee
holds for all $\g\in S$. That is, the modular condition can be lifted to an exact equality.

First we assume the Claim and complete the proof of \eqref{eq:spec3}. Let $q_{0}$ be a 
prime
 of size $\asymp T_{1}^{(\gd-\gT_{0})/2}$, say, such that $v_{*}\not\equiv0(\mod q_{0})$; then
\beann
|S|
&\ll&
\#
\{
\|\g_{1}\|<T_{1}:\<e_{1}
,\g v_{*}\>\equiv z_{*}(\mod q_{0})
\}
\\
&\ll&
q_{0}\left(
\frac1{q_{0}^{2}}
T_{1}^{\gd}
+
T_{1}^{\gT_{0}}
\right)
\ll
\frac1{q_{0}}
T_{1}^{\gd}
,
\eeann
by the argument in Case 1. 
Recall we assumed that $q<N$. Since $q_{0}$ above is  a small power of $N$,
 the above saves a tiny power of $q$, as desired.
\\

It remains to establish the Claim. 
For each $\g\in S$, consider the condition
$$
\<e_{1},\g\, v\>
=
\sum_{1\le j\le 4}\g_{1,j}\, v_{j}\equiv r(\mod q).
$$
First massage the equation into one with no trivial solutions.
Since $v$ is a primitive vector, after a linear change of variables we may assume that $(v_{1},q)=1$. 
Then multiply through by $\bar v_{1}$, where $v_{1}\bar v_{1}\equiv 1(\mod q)$, getting
\be\label{eq:3}
\g_{1,1}
+
\sum_{2\le j\le 4}\g_{1,j}\, v_{j}\bar v_{1}\equiv r \bar v_{1}(\mod q).
\ee
Now, for variables $V=(V_{2},V_{3},V_{4})$ and $Z$, and each $\g\in S$,
consider the (linear) polynomials $P_{\g}\in\Z[V,Z]$:
$$
P_{\g}(V,Z):=
\g_{1,1}
+
\sum_{2\le j\le 4}\g_{1,j}\, V_{j}-Z,
$$
and the affine variety
$$
\cV:=\bigcap_{\g\in S}\{P_{\g}=0\}.
$$
If this variety
$\cV(\C)$
 is non-empty, 
 then there is clearly a rational solution, 
$(V^{*},Z^{*})\in\cV(\Q)$. 
Hence
 we have found a rational solution  to \eqref{eq:wStar}, namely  $v^{*}=(1,V_{2}^{*},V_{3}^{*},V_{4}^{*})\neq0$ and $z^{*}=Z^{*}$. Since  \eqref{eq:wStar}  is 
homogeneous, we may clear denominators, getting an integral solution, $v_{*},z_{*}$.

Thus we henceforth assume by contradiction that the variety $\cV(\C)$ is empty. 
Then by Hilbert's Nullstellensatz, there are polynomials $Q_{\g}\in\Z[V,Z]$ and an integer $\fd\ge1$ so that
\be\label{eq:2}
\sum_{\g\in S}P_{\g}(V,Z)\ Q_{\g}(V,Z)=\fd,
\ee
for all $(V,Z)\in\C^{4}$. Moreover, 
Hermann's method \cite{Hermann1926} (see  \cite[Theorem IV]{MasserWustholz1983})
gives effective bounds on the heights of $Q_{\g}$ and $\fd$ in the above Bezout equation. Recall the height 
of a polynomial is the 
logarithm of its largest coefficient (in absolute value); thus the polynomials $P_{\g}$ are linear in four variables with 
height
$
\le\log T_{1}$. 
Then $Q_{\g}$ and $\fd$ can be found so that
\be\label{eq:fdIs}
\fd\le e^{8^{4\cdot 2^{4-1}-1}(\log T_{1}+8\log 8)}
\ll
T_{1}^{10^{28}}
.
\ee
(Much better bounds are known, see e.g.
 \cite[Theorem 5.1]{BerensteinYger1991}, but these suffice for our purposes.)
 
On the other hand, reducing \eqref{eq:2} modulo $q$ and evaluating at 
$$
V_{0}=(
v_{2}\bar v_{1},
v_{3}\bar v_{1},
v_{4}\bar v_{1}),
\qquad
Z_{0}=r\bar v_{1},
$$
we have
$$
\sum_{\g\in S}P_{\g}(V_{0},Z_{0}) Q_{\g}(V_{0},Z_{0})\equiv0\equiv\fd(\mod q),
$$
by \eqref{eq:3}. But then since $\fd\ge1$, we in fact have $\fd\ge q$, which is incompatible with \eqref{eq:fdIs} and \eqref{eq:qToT1}.
This furnishes our desired contradiction, completing the proof.
\epf

Next we need a slight generalization of Lemma \ref{lem:spec0}, which will be used in the major arcs analysis, see \eqref{eq:fMdecomp}.
\begin{lem}\label{lem:spec1}
Let $1<K\le T_{2}^{1/10}$, fix $|\gb|<K/N$, and fix $x,y\asymp X$. Then for any $\g_{0}\in\G$,
any $q\ge1$, and any group $\tilde\G(q)$ satisfying \eqref{eq:G1G}, we have
\bea
\nonumber
\sum_
{\g\in\fF\cap\{ \g_{0}\tilde\G(q)\}}
e
\bigg(
\gb\,
\ff_{\g}(2x,y)
\bigg)
&=&
{1\over [\G:\tilde\G(q)]}
\sum_
{\g\in\fF}
e
\bigg(
\gb\,
\ff_{\g}(2x,y)
\bigg)
\\
\label{eq:spec1}
&&
\hskip1in
+
O(
T^{\gT}
K)
,
\eea
where $\gT<\gd$ depends only on the spectral gap for $\G$, and the implied constant 
does
not 
depend
on 
$q$, $\g_{0}$, $\gb$, $x$ or $y$.
\end{lem}

\pf
The proof follows with minor changes that of Lemma \ref{lem:spec0}, so we give a sketch; see also \cite[\S4]{BourgainKontorovichSarnak2010}.

According to the construction \eqref{eq:fFdef} of $\fF$, the $\g$'s in question satisfy $\g=\g_{1}\g_{2}\in\g_{0}\tilde\G(q)$, and hence we can write
$$
\g_{2}=\g_{1}^{-1}\g_{0}\g_{2}',
$$
with $\g_{2}'\in\tilde\G(q)$. Then $\g_{2}'=\g_{0}^{-1}\g_{1}\g_{2}$, and using \eqref{eq:ffToInProd}, we can write the left hand side of \eqref{eq:spec1} as
$$
\sum_{\g_{1}\in\G\atop T_{1}<\|\g_{1}\|<2T_{1}}
\sum_{\g_{2}'\in\tilde\G(q)\atop T_{2}<\|\g_{1}^{-1}\g_{0}\g_{2}'\|<2T_{2}}
\bo_{\{\<e_{1},\g_{0}\g_{2}'\,v_{0}\>>T/100\}}\
e
\bigg(
\gb\,
\<w_{x,y},
\g_{0}\g_{2}'\,v_{0}\>
\bigg)
.
$$
Now we fix $\g_{1}$ and mimic the proof of Lemma \ref{lem:spec0} in $\g_{2}'$. 

Replace \eqref{eq:fgIs} by
$$
f(g):=
\bo_{\{T_{2}<\|\g_{1}^{-1}g\|<2T_{2}\}}
\bo_{\{\<e_{1},g\,v_{0}\>>T/100\}}\
e
\bigg(
\gb\,
\<w_{x,y},
g\,v_{0}\>
\bigg)
.
$$
Then \eqref{eq:FqIs}-\eqref{eq:reg1Lem} remains essentially unchanged, save cosmetic changes such as replacing \eqref{eq:cNq} by
$F_{q}(\g_{1}\g_{0}^{-1},e)$. Then  in the estimation of the difference $|\cN_{q}-\cH_{q}|$ by splitting the sum on $\g_{2}'$ into ranges, the argument now proceeds as follows.
\begin{enumerate}
\item
The range \eqref{eq:reg1Lem} should be replaced by 
$$
\|\g_{1}\g_{0}^{-1}\g_{2}'\|<T_{2}(1-10\eta),\text{ or }
\|\g_{1}\g_{0}^{-1}\g_{2}'\|>2T_{2}(1+10\eta),
$$
$$
\text{ or }\<e_{1},\g_{1}\g_{0}^{-1}\g_{2}'\,v_{0}\><\frac T{100}(1-10\eta).
$$

\item
The range  \eqref{eq:reg2Lem} should be replaced by the 
range 
$$
T_{2}(1+10\eta)<\|\g_{1}\g_{0}^{-1}\g_{2}'\|<2T_{2}(1-10\eta),
\text{ and }\<e_{1},\g_{1}\g_{0}^{-1}\g_{2}'\,v_{0}\>>\frac T{100}(1+10\eta),
$$
in which $f$ is differentiable.
Here instead of the difference $|f(\g_{1}\g_{0}^{-1}g\g_{2}'h)-f(\g_{1}\g_{0}^{-1}\g_{2}')|$ vanishing, it is now bounded by
$$
\ll \eta K,
$$
for a net contribution to the error of $\ll \eta K T^{\gd}$.

\item
In the remaining range, \eqref{eq:reg3Lem} remains unchanged, using $|f|\le1$.
\end{enumerate}

The error in \eqref{eq:NqErr} is then replaced by
$$
O(\eta\, K\, T_{2}^{\gd} + T_{2}^{\gd-\vep}\eta^{-10}).
$$
Optimizing $\eta$ and renaming $\gT$ gives the bound $O(T_{2}^{\gT}K^{10/11})$, which is better than claimed in the power of $K$. Rename $\gT$ once more using \eqref{eq:TT1T2} and \eqref{eq:cCis}, giving \eqref{eq:spec1}.
\epf

The following is our last counting lemma, showing a certain equidistribution among the values of $\ff_{\g}(2x,y)$ at the scale $N/K$. This bound is used in the major arcs, see
the proof of Theorem  \ref{thm:cMNis}.

\begin{lem}
\label{lem:spec2}
Fix $N/2<n<N$,  $1<K\le T_{2}^{1/10}$, and $x,y\asymp X$. Then
\be\label{eq:spec2}
\sum_
{\g\in
\fF
}
\bo_{
\big\{
|
\ff_{\g}(2x,y)-
n|
<
\frac N{K}
\big\}
}
\gg
{T^{\gd}\over K}
+
T^{\gT}
,
\ee
where $\gT<\gd$ only depends on the spectral gap for $\G$. The implied constant is independent of $x,y,$ and $n$. 
\end{lem}

\pf[Sketch]
The proof is an 
explicit calculation nearly identical to the one given in \cite[\S5]{BourgainKontorovichSarnak2010}; 
we give only  a sketch here. 
Write the left hand side of \eqref{eq:spec2} as
$$
\sum_{\g_{1}\in\G\atop T_{1}<\|\g_{1}\|<2T_{1}}
\sum_{\g_{2}\in\G\atop T_{2}<\|\g_{2}\|<2T_{2}}
\bo_{\{
\<e_{1},\g_{1}\g_{2}v_{0}\>
>
T/100
\}}
\bo_{\{
|\<w_{x,y},\g_{1}\g_{2}v_{0}\>-n|<N/K
\}}
.
$$
Fix $\g_{1}$ and express the condition on $\g_{2}$ as $\g_{2}\in R\subset G$, where $R$ is the region
$$
R=
R_{\g_{1},x,y,n}:=
\left\{
g\in G
:
\begin{array}{c}
T_{2}<\|g\|<2T_{2}
\\
\<\g_{1}^{t}e_{1},g\, v_{0}\>>T/100
\\
|\<\g_{1}^{t}w_{x,y},
g\,
v_{0}
\>
-n|
<
\frac N{K}
\end{array}
\right\}
.
$$
Lift $G=\SO_{F}(\R)$ to its spin cover $\tilde G=\SL_{2}(\C)$ via the map $\iota$ of \eqref{eq:iota}. Let $\tilde R\subset\tilde G$ be the corresponding pullback region, and decompose $\tilde G$ into Cartan $KAK$ coordinates according to \eqref{eq:gDecomp}. 
Note that $\iota$ is quadratic in the entries, so, e.g., the condition 
\be\label{eq:gToIg}
\text{$\|g\|^{2}\asymp T$ gives $\|\iota(g)\|\asymp T$, }
\ee
explaining the factor $\|a(g)\|^{2}$ appearing in \eqref{eq:Vin}. 

Then chop $\tilde R$ 
 into 
 spherical caps
 and apply Theorem \ref{thm:Vin}. 
 The same argument as in \cite[\S5]{BourgainKontorovichSarnak2010} then leads to \eqref{eq:spec2}, after renaming $\gT$; we suppress the details.
\epf

\newpage

\subsection{Local Analysis Statements}
\

In this subsection, we study a certain exponential sum which arises in a crucial way in our estimates. 
Fix $\ff\in\fF$, and write $\ff=f-a$ with
$$
f(x,y)=Ax^{2}+2Bxy+Cy^{2}
$$ 
according to \eqref{eq:ffVdef}. 
Let
$q_{0}\ge1$, 
fix $r$ with  $(r,q_{0})=1$, and fix $n,m\in\Z$. 
(The notation is meant to be consistent with its later use; there will be another parameter $q$, and $q_{0}$ will be a divisor of $q$.)
Define the exponential sum
\be\label{eq:cSfDef}
\cS_{f}(q_{0},r;n,m)
:=
{1\over q_{0}^{2}}
\sum_{k(q_{0})}
\sum_{\ell( q_{0})}
e_{q_{0}}
\bigg(
r
f(k,\ell)
+nk+m\ell
\bigg)
.
\ee
This sum appears naturally in many places in the minor arcs analysis, see e.g. \eqref{eq:cRfuIs} and \eqref{eq:RfuNew}.
Our first lemma is completely standard, see, e.g. \cite[\S12.3]{IwaniecKowalski}.
\begin{lem}
With the above conditions,
\be\label{eq:cSfbnd1}
|
\cS_{f}(q_{0},r;n,m)
|
\le q_{0}^{-1/2}
.
\ee
\end{lem}

\begin{rmk}
Being a sum in two variables, one might expect square-root cancellation in each, giving a savings of $q_{0}^{-1}$; indeed this is what we obtain, modulo some coprimality conditions, see \eqref{eq:SfEval}. For some of our applications, saving just one square-root is plenty, and we can ignore the coprimality; hence the cleaner statement in \eqref{eq:cSfbnd1}.
\end{rmk}
\pf
Write $\cS_{f}$ for $\cS_{f}(q_{0},r;n,m)$.
Note first that 
$\cS_{f}$
is multiplicative in $q_{0}$, so we study the case $q_{0}=p^{j}$ is a prime power. 
Assume for simplicity $(q_{0},2)=1$; similar calculations are needed to handle the $2$-adic case.

First we re-express $\cS_{f}$ in a 
more convenient form.
By
Descartes theorem \eqref{eq:Fv}, primitivity of the 
gasket
 $\sG$, and \eqref{eq:ABCdef}, we have that $(A,B,C)=1$; 
 assume henceforth that $(C,q_{0})=1$, say. 
Write $\bar x$ for the multiplicative inverse of $x$ (the modulus will be clear from context). Recall throughout that $(r,q_{0})=1$. 

Looking at the terms in the summand of $\cS_{f}$,  we have
\beann
&&
\hskip-.5in
r
f(k,\ell)
+
nk+m\ell
\quad(\mod q_{0})
\\
&\equiv&
r
(Ak^{2}+2Bk\ell+C\ell^{2})
+
nk+m\ell
\\
&\equiv&
rC(\ell
+B\bar Ck
)^{2}
+
r\bar Ck^{2}
(AC
-
B^{2}
)
+
nk+m\ell
\\
&\equiv&
rC(\ell
+B\bar Ck
)^{2}
+
a^{2}
r\bar C
k^{2}
+
nk+m\ell
\\
&\equiv&
rC(\ell
+B\bar Ck
+
\overline{2rC}
m
)^{2}
-\overline{4rC} m^{2}
+
a^{2}
r\bar C
k^{2}
+
k(n
-
B\bar C m)
,
\eeann 
where we used \eqref{eq:disc}. Hence  we have
\beann
\cS_{f}
&=&
{1\over q_{0}^{2}}
e_{q_{0}}
(
-\overline{4rC} m^{2}
)
\sum_{k( q_{0})}
e_{q_{0}}
\bigg(
a^{2}
r\bar C
k^{2}
+
k(n
-
B\bar C m)
\bigg)
\\
&&
\qquad
\times
\sum_{\ell( q_{0})}
e_{q_{0}}
\bigg(
rC(\ell
+B\bar Ck
+
\overline{2rC}
m
)^{2}
\bigg)
,
\eeann
and the $\ell$ sum is just a classical Gauss sum. It can be evaluated explicitly, see e.g. \cite[eq. (3.38)]{IwaniecKowalski}. Let 
$$
\vep_{q_{0}}:=\twocase{}1{if $q_{0}\equiv 1(\mod 4)$}i{if $q_{0}\equiv3(\mod 4)$.}
$$
Then the Gauss sum on $\ell$ is $\vep_{q_{0}}\sqrt q_{0}\left({rC\over q_{0}}\right)$, where $({\cdot\over q_{0}})$ is the Legendre symbol. Thus we have
\beann
\cS_{f}
&=&
{\vep_{q_{0}}\over q_{0}^{3/2}}
\left({rC\over q_{0}}\right)
e_{q_{0}}
(
-\overline{4rC} m^{2}
)
\sum_{k( q_{0})}
e_{q_{0}}
\bigg(
a^{2}
r\bar C
k^{2}
+
k(n
-
B\bar C m)
\bigg)
.
\eeann

Let
\be\label{eq:tildQdef}
\tilde q_{0}:=(a^{2},q_{0}), \qquad q_{1}:=q_{0}/\tilde q_{0},\qquad\text{and }\qquad a_{1}:=a^{2}/\tilde q_{0},
\ee 
so that $a^{2}/q_{0}=a_{1}/q_{1}$ in lowest terms. 
Break the sum on $0\le k< q_{0}$ according to $k= k_{1}+q_{1}\tilde k$,
with $0\le k_{1}< q_{1}$ and $0\le\tilde k< \tilde q_{0}$. Then
\beann
\cS_{f}
&=&
{\vep_{q_{0}}\over q_{0}^{3/2}}
\left({rC\over q_{0}}\right)
e_{q_{0}}
(
-\overline{4rC} m^{2}
)
\\
&&
\qquad
\times
\sum_{k_{1}( q_{1})}
e_{q_{1}}
\bigg(
a_{1}
r\bar C
(k_{1})^{2}
\bigg)
e_{q_{0}}
\bigg(
{k_{1}
}
(n
-
B\bar C m)
\bigg)
\\
&&
\qquad
\times
\sum_{\tilde k( \tilde q_{0})}
e_{\tilde q_{0}}
\bigg(
{\tilde k}
(n
-
B\bar C m)
\bigg)
.
\eeann
The last sum vanishes unless $n-B\bar Cm\equiv0(\mod \tilde q_{0})$, in which case it is $\tilde q_{0}$. In the latter case, define $L$ by 
\be\label{eq:Ldef}
L:=(Cn-Bm )/ \tilde q_{0} .
\ee
Then we have
\beann
\cS_{f}
&=&
\bo_{
nC\equiv mB( \tilde q_{0})
}
{\vep_{q_{0}}\over q_{0}^{3/2}}
\left({rC\over q_{0}}\right)
e_{q_{0}}
(
-\overline{4rC} m^{2}
)
\\
&&
\qquad
\times
e_{q_{1}}
\bigg(
-
\overline{4a_{1}rC}
L^{2}
\bigg)
\left[
\sum_{k_{1}( q_{1})}
e_{q_{1}}
\bigg(
a_{1}
r\bar C
(
k_{1}
+
\overline{2a_{1}r}
L
)^{2}
\bigg)
\right]
\tilde q_{0}
.
\eeann
The Gauss sum in brackets is again evaluated as 
$
\vep_{q_{1}}
q_{1}^{1/2}
\left(
{
a_{1}
r\bar C
\over q_{1}
}
\right)
,
$
so we have
\bea
\label{eq:SfEval}
\cS_{f}(q_{0},r;n,m)
&=&
\bo_{
nC\equiv mB(\tilde q_{0})
}
{
\vep_{q_{0}}\vep_{q_{1}}
\tilde q_{0}
^{1/2}
\over q_{0}
}
e_{q_{0}}
(
-\overline{4rC} m^{2}
)
\\
\nonumber
&&
\qquad
\times
e_{q_{1}}
\bigg(
-
\overline{4a_{1}rC}
L^{2}
\bigg)
\left({rC\over q_{0}}\right)
\left(
{
a_{1}
r\bar C
\over q_{1}
}
\right)
.
\eea
The claim then follows trivially.
\epf
\

Next we introduce a certain average of a pair of such sums. Let $f, q_{0}, r, n,$ and $m$ be as before, and fix $q\equiv0(\mod q_{0})$ 
 and 
 $(u_{0},q_{0})=1$. Let $\ff'\in\fF$ be another 
shifted form $\ff'=f'-a'$, with 
$$
f'(x,y)=A'x^{2}+2B'xy+C'y^{2}.
$$ 
Also let $n',m'\in\Z$. Then define
\bea\label{eq:cSdef}
\hskip-.5in
\cS&=&
\cS(q,q_{0},f,f',n,m,n',m';u_{0})
\\
\nonumber
&:=&
\sideset{}{'}
\sum_{r(q)}
\cS_{f}(q_{0},r\d_{0};n,m)
\overline{\cS_{f'}(q_{0},r\d_{0};n',m')}
e_{q}(r (a'- a))
.
\eea
This sum also
appears naturally
in the minor arcs analysis, see \eqref{eq:IQbnd1} and \eqref{eq:IQbnd2}.

\begin{lem}
With the above  notation, 
we have the estimate
\be
\label{eq:modBndaNeqAp}
|\cS|
\ll
\left({q/ q_{0}}\right)^{2}
{
\{(a^{2},q_{0})
\cdot
((a')^{2},q_{0})\}
^{1/2}
\over q^{5/4}
}
(a-a',q)^{1/4}
.
\ee
\end{lem}
\begin{rmk}
Treating all $\gcd$'s above as $1$ and pretending $q=q_{0}$, the trivial bound here (after having saved essentially a whole $q$ from each of the two $\cS_{f}$ sums) is $1/q$, since the $r$ sum is unnormalized. So \eqref{eq:modBndaNeqAp} saves an extra $q^{1/4}$ in the $r$ sum. (In fact we could have saved the expected $q^{1/2}$, but this does not improve our final estimates.)
\end{rmk}
\pf
Observe  that $\cS$ is multiplicative in $q$, so we again consider the prime power case $q=p^{j}$, $p\neq2$; then $q_{0}$ is also a prime power, since $q_{0}\mid q$. As before, we may assume $(C,q_{0})=(C',q_{0})=1$.

Recall $a_{1}$, $\tilde q_{0}$, 
and $L$ 
given in \eqref{eq:tildQdef}
and \eqref{eq:Ldef}, 
and let $a_{1}'$, $\tilde q_{0}'$ and $L'$ 
be defined similarly.
Inputting the analysis from
 \eqref{eq:SfEval} into both $\cS_{f}$ and $\cS_{f'}$, 
we have
\bea
\nonumber
\cS
&=&
\bo_{
nC\equiv 2mB(\tilde q_{0})
\atop
n'C'\equiv 2m'B'(\tilde q'_{0})
}
{
\vep_{q_{1}}
\bar\vep_{q_{1}'}
(\tilde q_{0}\tilde q'_{0})^{1/2}
\over q_{0}^{2}
}
\left({
CC'\over q_{0}}\right)
\left(
{
a_{1}
u_{0}\bar C
\over q_{1}
}
\right)
\left(
{
a_{1}'
u_{0}\bar C'
\over q_{1}'
}
\right)
\\
\label{eq:cS1}
&&
\times
\Bigg[
\sideset{}{'}
\sum_{r(q)}
\left(
{
r\over q_{1}
}
\right)
\left(
{
r\over q_{1}'
}
\right)
e_{q}(r \{a'- a\})
\\
\nonumber
&&
\qquad
\times
e_{q_{0}}
\Bigg(
\overline{ 4
 r
 u_{0}
 }
\bigg\{
\overline{C'} (m')^{2}
-\overline{C} m^{2}
+
\overline{a_{1}'C'}
(L')^{2}
\tilde q'
-
\overline{a_{1}C}
L^{2}
\tilde q
\bigg\}
\Bigg)
\Bigg]
.
\eea

The term in brackets $\big[\cdot\big]$  is a Kloosterman- or Sali\'e-type sum, for which we have an elementary bound \cite{Kloosterman1927} to the power $3/4$:
%
%
\bea
\nonumber
|\cS|
&\ll&
{
(\tilde q_{0}
\tilde q'_{0})
^{1/2}
\over q_{0}^{2}
}
q^{3/4}
(a-a',q)^{1/4}
,
\eea
giving 
the claim. (There is no improvement in our use of this estimate from appealing to Weil's bound instead of Kloosterman's; any
 power gain suffices).
\epf

In the case $a=a'$, \eqref{eq:modBndaNeqAp} only saves one power of $q$, and in \S\ref{sec:QXT} we will need slightly more; see the proof of \eqref{eq:cI42bnd}.
We get a bit more cancellation in the special case $f(m,-n)\neq f'(m',-n')$ below.

\begin{lem}
Assuming
$a=a'$ and $f(m,-n)\neq f'(m',-n')$,
we have the estimate
\be
\label{eq:modBndaEqAp}
|\cS|
\ll
(q/q_{0})^{5}\
{
(a^{2},q_{0})
\over q^{9/8}
}
\cdot
|f(m,-n)-f'(m',-n')|^{1/2}
.
\ee
\end{lem}


\pf
Assume first that $q$ (and hence $q_{0}$) 
is
a prime power, continuing to omit the prime $2$.
Returning to the
definition of $\cS$ in \eqref{eq:cSdef},
it is clear in the case $a=a'$  that
$$
\sum_{r(q)}'=(q/q_{0})\sum_{r(q_{0})}'.
$$
Hence we again apply Kloosterman's $3/4$th bound to \eqref{eq:cS1}, getting
\bea
\label{eq:modBndaEqAp1}
|\cS
|
&\ll&
\bo_{
nC\equiv 2mB(\tilde q_{0})
\atop
n'C'\equiv 2m'B'(\tilde q'_{0})
}
(q/q_{0})^{9/2}
{
(a^{2},q_{0})
\over q^{5/4}
}
\\
\nonumber
&&
\times
\prod_{p^{j}\| q_{0}}
\Bigg(
p^{j}
,
\bar 4
\left\{
\overline{C'} (m')^{2}
-\overline{C} m^{2}
+
\overline{
a_{1}}
(a^{2},p^{j})
(
\overline{C'}
(L')^{2}
-
\overline{C}
L^{2}
)
\right\}
\Bigg)^{1/4}
,
\eea
which is valid now without the assumption
that $q_{0}$ is a prime power. (Here $a_{1}$ satisfies $a^{2}=a_{1}(a^{2},p^{j})$ as in \eqref{eq:tildQdef}, and $L$ is given in \eqref{eq:Ldef}, so both depend on $p^{j}$.)

Break the primes  diving $q_{0}$ into two sets, $\cP_{1}$ and $\cP_{2}$, defining  $\cP_{1}$ to be the set of those primes $p$ for which 
\be\label{eq:pr2}
\overline{
C} m^{2}
+
\overline{
C}
L^{2}
\overline{
a_{1}}
(a^{2},p^{j})
\equiv 
\overline{
C'} (m')^{2}
+
\overline{
C'}
(L')^{2}
\overline{
a_{1}}
(a^{2},p^{j})
\ \
(\mod p^{\lceil j/2\rceil})
,
\ee
and $\cP_{2}$ the rest. 
For the latter, the $\gcd$ in $(p^{j},\cdots)$ of \eqref{eq:modBndaEqAp1} is at most $p^{j/2}$, so we clearly have
\be\label{eq:cP2}
\prod_{p^{j}\| q_{0}\atop p\in\cP_{2}}(p^{j},\cdots)^{1/4}
\le
\prod_{p^{j}\| q_{0}
}p^{j/8}
=
q_{0}^{1/8}
.
\ee

For $p\in\cP_{1}$, we multiply both sides of \eqref{eq:pr2} by 
$$
a^{2}=
AC-B^{2}
=
A'C'-(B')^{2}
=
a_{1}(a^{2},p^{j})
,
$$ 
giving
\bea
\label{eq:15}
&&
\hskip-.5in
(AC-B^{2})\overline{C} m^{2}
+
\overline{C}
L^{2}
(a^{2},p^{j})^{2}
\\
\nonumber
&
\equiv
& 
(A'C'-(B')^{2})\overline{C'} (m')^{2}
+
\overline{C'}
(L')^{2}
(a^{2},p^{j})^{2}
\qquad(\mod p^{\lceil j/2\rceil})
.
\eea
Using \eqref{eq:Ldef} that
$$
nC- 
mB = (a^{2},p^{j}) L
,
\qquad
n'C'- 
m'B' = (a^{2},p^{j}) L'
$$
and subtracting $a$ from both sides of \eqref{eq:15}, 
we have shown that
\be
\label{eq:fmnfp}
f'(m',-n')
\equiv
f(m,-n)
\qquad(\mod p^{\lceil j/2\rceil})
.
\ee
Let 
$$
Z=|f(m,-n)-f'(m',-n')|
.
$$
By assumption $Z\neq0$. Moreover \eqref{eq:fmnfp} implies that
$$
\left(
 \prod_{p\in\cP_{1}}p^{\lceil j/2\rceil}
 \right)
 \mid
 Z
,
$$ 
and hence
\be\label{eq:cP1}
\prod_{p^{j}\|q_{0}\atop p\in\cP_{1}}p^{j/4}
\le
Z^{1/2}
.
\ee
Combining \eqref{eq:cP1} and \eqref{eq:cP2} in \eqref{eq:modBndaEqAp1} gives the claim.
%
\epf

Finally we need some  savings in the case $a=a'$ and $f(m,-n)=f'(m',-n')$. This will no longer come from $\cS$ itself, but from the following supplementary lemmata. 

\begin{lem}
Fix an equivalence class $\cK$ of primitive binary quadratic forms of discriminant $-4a^{2}$. 
We claim that the number of equivalent forms $f\in\cK$ with $\ff=f-a\in\fF$ is bounded, that is,
\be\label{eq:fincK}
\#\{\ff\in\fF:f\in\cK\}=O(1).
\ee
\end{lem}
\pf
From \eqref{eq:ABCdef}, \eqref{eq:fFdef}, and \eqref{eq:disc}, we have that
 $f(m,n)=Am^{2}+2Bmn+Cn^{2}$ has coefficients of size
$$
A,B,C\ll T,
$$
and  $AC-B^{2}=a^{2}$, with
$
a\asymp T.
$ 
It follows that $AC\asymp T^{2}$, and hence 
\be\label{eq:ACisT}
A,C\asymp T.
\ee

Now suppose we have $\ff=f-a$ and $\ff'=f'-a$ with $f$ as above and $f'$ having coefficients $A',B',C'$. If $f$ and $f'$ are equivalent then there is an element $\mattwo ghij\in\GL(2,\Z)$ so that
\bea\label{eq:ApBpCpABC}
A' &=& g^{2} A + 2 gi B + i^{2} C,\\
\nonumber
B' &=&  gh A + (gj+hi) B +  ij C ,\\
\nonumber
C' &=&  h^{2} A + 2 hj B + j^{2} C.
\eea
The first line can be rewritten as 
$$
A'
=
C
( i
+ g B /C
)^{2}
 +
 g^{2} {4a^{2}\over C}
,
$$
so that
$$
g^{2}
\le
A' 
 {C\over4a^{2}}
\ll
1
 .
$$
Similarly,
$$
( i
+ g B /C
)^{2}
\le
{A' \over C}
\ll
1
,
$$
and 
hence
$|i|\ll 1$. In a similar fashion, we see that $|h|$ and $|j|$ are also bounded, thus the number of equivalent forms in $\cK$ is bounded, as claimed.
\epf

\begin{lem}
For a fixed large integer $z$, the number of inequivalent classes $\cK$ of primitive quadratic forms 
of determinant $-4a^{2}
$ which represent $z$ is 
\be\label{eq:cKzBnd}
\ll_{\vep}\
z^{\vep}\cdot 
(z,4a^{2})^{1/2}
,
\qquad
\text{for any $\vep>0$.} 
\ee
\end{lem}
\pf
If $f\in\cK$ represents $z$, say $f(m,n)=z$, then, setting $w=(m,n)$, $f$ represents $z_{1}:=z/w^{2}$ primitively. 
We see from \eqref{eq:ApBpCpABC} that $f$ is then in the same class as $f_{1}(m,n)=z_{1}m^{2}+2Bmn+Cn^{2}$, with
$$
-4a^{2}=z_{1}C-B^{2}
.
$$ 
Moreover, by a unipotent change of variables preserving $z_{1}$, we can force $B$ into the range $[0,z_{1})$, that is, $B$ is determined mod $z_{1}$. So the number of inequivalent  such $f_{1}$ is equal to
\be\label{eq:BtoZ1}
\#\{B(\mod z_{1}):B^{2}\equiv-4a^{2}(z_{1})\}
=
\prod_{p^{e}\mid\mid z_{1}}
\#\{B^{2}\equiv-p^{2f}(p^{e})\}
,
\ee
where $p^{f}\mid\mid 2a$. If $2f\ge e$, then the number of local solutions is at most $p^{e/2}$. Otherwise, write $B=B_{1}p^{f}$; then there are at most $2$ solutions to $B_{1}^{2}\equiv-1(\mod p^{e-2f}),$
and there are $p^{f}$ values for $B$ once $B_{1}$ is determined. Hence the number of local solutions is at most $2\cdot\min(p^{e/2},p^{f})$, so the number of solutions to \eqref{eq:BtoZ1} is at most
$$
2^{\gw(z)}(z_{1},4a^{2})^{1/2}\ll_{\vep}z^{\vep}(z,4a^{2})^{1/2}.
$$
The number of divisors $z_{1}$ of $z$ is $\ll_{\vep} z^{\vep}$, completing the proof.
\epf

\begin{lem} \label{lem:kEllD}
Fix $(A,B,C)=1$ and $d\mid AC-B^{2}$. Then there are integers  $k,\ell
$ with $(k,\ell,d)=1$ so that, whenever $Am^{2}+2Bmn+Cn^{2}\equiv0(d)$, we have
\be\label{eq:nkmell}
(mk+n\ell)^{2}\equiv0(d).
\ee
\end{lem}
\pf
We will work locally, then lift to a global solution.  Let $p^{e}\mid\mid d$. 

Case 1: If $(p,A)=1$, then $Am^{2}+2Bmn+Cn^{2}\equiv0(p^{e})$ implies
$$
(m+\bar A Bn)^{2}-\bar A ^{2}B^{2}n^{2}+\bar ACn^{2}\equiv
(m+\bar A Bn)^{2}\equiv
0(p^{e}).
$$
In this case, we set $k_{p}:=1$, and $\ell_{p}:={\bar A B}$. 

Case 2:
If $(p,A)>1$, then by primitivity, $(p,C)=1$. As before, we have
$
(n+\bar C Bm)^{2}\equiv
0(p^{e}),
$
and we choose $k_{p}={\bar C B}$, $\ell_{p}:=1$.

By the Chinese Remainder Theorem, there are integers $k$ and $\ell$ so that $k\equiv k_{p}(\mod p^{e})$, and similarly with $\ell$. By construction, we have $(k,\ell,d)=1$,
as claimed.
\epf

\begin{lem}
Given large $M$, $(A,B,C)=1$ and $d\mid AC-B^{2}$,
\be\label{eq:nmUpToM}
\#\{m,n<M:Am^{2}+2Bmn+Cn^{2}\equiv0(d)\}
\ee
$$
\ll_{\vep} 
d^{\vep}
\left({M^{2}\over d^{1/2}}+M\right)
.
$$
\end{lem}
\pf
As in Lemma \ref{lem:kEllD}, $A,B,C$ and $d$ determine   $k,\ell$ so that
$$
\sum_{m,n<M}\bo_{\{Am^{2}+2Bmn+Cn^{2}\equiv0(d)\}}
\le
\sum_{m,n<M}\bo_{\{(mk+n\ell)^{2}\equiv0(d)\}}
.
$$
But then there is a $d_{1}\mid d$, with $d\mid d_{1}^{2}$ so that $mk+n\ell\equiv0(d_{1})$.
Let $w=(\ell, d_{1})$; then $mk\equiv0(w)$ implies $m\equiv0(w)$ since $(k,\ell,d)=1$. There are at most $1+M/w$ such $m$ up to $M$. With $m$ fixed, $n$ is uniquely determined mod $d_{1}/w$.
Hence we get
the bound 
\beann
\eqref{eq:nmUpToM}
&\le&
\sum_{d_{1}\mid d\atop d\mid d_{1}^{2}}
\sum_{w\mid d_{1}}
\sum_{m,n<M}\bo_{\{m\equiv0(\mod w)\}}\bo_{\{ n\equiv-\overline{\frac \ell w}\frac mw k(\mod \frac{d_{1}}w)\}}
\\
&\ll&
\sum_{d_{1}\mid d\atop d\mid d_{1}^{2}}
\sum_{w\mid d_{1}}
\left(
{M\over w}
+1
\right)
\left(
{wM\over d_{1}}
+1
\right)
\ll_{\vep}
d^{\vep}
\left(
{M^{2}\over d^{1/2}}
+M
\right)
,
\eeann
as claimed.
\epf

Finally we collect the above lemmata into  our desired estimate, essential in the proof of \eqref{eq:cI41bnd}.
\begin{prop}
For large $M$ and  $\ff=f-a\in\fF$ fixed,
\be\label{eq:eqeq}
\#
\left\{
\begin{array}{c}
\ff'\in\fF
\\
m,n,m',n'<M
\end{array}
\Bigg| 
\begin{array}{c}
a'=a\\
\ff(m,-n)=\ff'(m',-n')
\end{array}
\right\}
\ll_{\vep}
(TM)^{\vep}
\left(
M^{2}
+
TM
\right)
,
\ee
for any $\vep>0$.
\end{prop}

\pf
Once $f,m,n$, and $\ff'=f'-a\in\fF$ are determined, it is elementary that there are  $\ll_{\vep} M^{\vep}$ values of $m',n'$ with $f(m,-n)=f'(m',-n')$. 
Decomposing $f'$ into classes and applying 
 \eqref{eq:fincK},
 \eqref{eq:cKzBnd}, 
and 
 \eqref{eq:nmUpToM},
in succession, we have
\beann
&&
\hskip-1in
\sum_{m,n<M}
\sum_{\ff'\in\fF\atop a'=a}
\sum_{m',n'<M}
\bo_{\{f(m,-n)=f'(m',-n')\}}
\\
&\ll_{\vep}&
\sum_{m,n<M}
\sum_{\ff'\in\fF\atop a'=a}
\bo_{\{f'\text{ represents }f(m,-n)\}}
M^{\vep}
\\
&\ll&
M^{\vep}
\sum_{m,n<M}
\sum_{\text{classes }\cK\atop\text{representing }f(m,-n)}
\sum_{\ff'\in\fF\atop a'=a,f'\in\cK}
1
\\
&\ll_{\vep}&
(TM)^{\vep}
\sum_{m,n<M}
(f(m,-n),4a^{2})^{1/2}
\\
&\ll&
(TM)^{\vep}
\sum_{d\mid 4a^{2}}
d^{1/2}
\sum_{m,n<M}
\bo_{\{f(m,-n)\equiv0(d)\}}
\\
&\ll&
(TM)^{\vep}
\sum_{d\mid 4a^{2}}
d^{1/2}
\left(
{M^{2}\over d^{1/2}}
+
M
\right)
\\
&\ll&
(TM)^{\vep}
\left(
M^{2}
+
M
a
\right)
,
\eeann
from which the claim follows since $a\ll T$.
\epf

\newpage


\section{Major Arcs}\label{sec:Maj}


We return to the setting and notation of \S\ref{sec:outline} with the goal of establishing \eqref{eq:cMNUbnd}. Thanks to the counting lemmata in \S\ref{sec:count}, we can now define the major ars parameters $Q_{0}$ and $K_{0}$ from \eqref{eq:Q0K0intro}. First recall the two numbers $\gT<\gd$ appearing  in \eqref{eq:spec1}, \eqref{eq:spec2}, and define
\be\label{eq:gT1def}
1<\gT_{1}<\gd
\ee
to be the larger of the two. Then set
\be\label{eq:Q0is}
Q_{0}= T ^{(\gd-\gT_{1})/20},
\qquad
K_{0}=Q_{0}^{2}.
\ee 
We may now also set the parameter $U$ from \eqref{eq:Uis} to be
\be\label{eq:fuIs}
U=Q_{0}{}^{(\eta_{0})^{2}/100},
\ee
where $0<\eta_{0}<1$ is the number which appears in Lemma \ref{lem:spec3}.

Let $\cM_{N}^{(U)}(n)$ denote either $\cM_{N}(n)$ or $\cM_{N}^{U}(n)$ from \eqref{eq:cMNdef}, \eqref{eq:cMNUdef}, respectively.
Putting \eqref{eq:hatFuncN} and \eqref{eq:cRNhatIs} (resp. \eqref{eq:cRNhatUis}) into \eqref{eq:cMNdef} (resp. \eqref{eq:cMNUdef}),
making a change of variables $\gt=r/q+\gb$, and unfolding the integral from $\sum_{m}\int_{0}^{1}$ to $\int_{\R}$ 
gives
\be\label{eq:cMNUn}
\cM_{N}^{(U)}(n)
=
\sum_{x,y\in\Z}
\gU\left(\frac {2x}X\right)
\gU\left(\frac yX\right)
\cdot
\fM(n)
\cdot
\sum_{u}
\mu(u)
,
\ee
where in the last sum, $u$ ranges over $u\mid (2x,y)$ (resp. and $u<U$). Here we have defined
\bea
\label{eq:fMis}
&&
\hskip-.5in
\fM(n)
=
\fM_{x,y}(n)
\\
\nonumber
&:=&
\sum_{q<Q_{0}}
\sideset{}{'}\sum_{r(q)}
\sum_{\g\in\fF}
e_{q}(r(\<w_{x,y},\g v_{0}\>-n))
\int_{\R}
\ft
\left(
\frac N{K_{0}}
\gb
\right)
e(\gb (\ff_{\g}(2x,y)-n))
d\gb
,
\eea
using \eqref{eq:ffToInProd}.

As in \eqref{eq:G0q}, let $\tilde\G(q)$ be the stabilizer of $v_{0}(\mod q)$. Decompose the sum on  $\g\in\fF$ in \eqref{eq:fMis} as 
a sum on ${\g_{0}\in\G/\tilde\G(q)}$ and ${\g\in\fF\cap\g_{0}\tilde\G(q)}$.
Applying Lemma \ref{lem:spec1} to the latter sum, using the definition of $\gT_{1}$ in \eqref{eq:gT1def}, and recalling the estimate \eqref{eq:G0qBnd}  gives
\be\label{eq:fMdecomp}
\fM(n)
=
\fS_{Q_{0}}(n)\cdot
\fW(n)
+
O\left({T^{\gT_{1}}\over N}K_{0}^{2}Q_{0}^{4}\right)
,
\ee
where
\beann
\fS_{Q_{0}}(n)
&:=&
\sum_{q<Q_{0}}
\sideset{}{'}\sum_{r(q)}
\sum_{\g_{0}\in\G/\tilde\G(q)}
{e_{q}(r(\<w_{x,y},\g_{0} v_{0}\>-n))
\over
[\G:\tilde\G(q)]
}
,
\\
\fW(n)
&:=&
{K_{0}\over N}\,
\sum_{\ff\in\fF}
\hat\ft
\left(
(\ff(2x,y)-n)
\frac {K_{0}}N
\right)
.
\eeann

Clearly we have thus split $\fM$ into  ``modular'' and ``archimedean'' components. It is now a simple matter to prove the following
\begin{thm}\label{thm:cMNis}
For $\foh N<n<N$, there exists a function $\fS(n)$ as in Theorem \ref{thm:RNU} so that
\be\label{eq:cMNbnd}
\cM_{N}(n)\gg\fS(n)T^{\gd-1}
.
\ee
\end{thm}
\pf
First we discuss the modular component. Write $\fS_{Q_{0}}$ as
$$
\fS_{Q_{0}}(n)
=
\sum_{q<Q_{0}}
{1\over
[\G:\tilde\G(q)]
}
\sum_{\g_{0}\in\G/\tilde\G(q)}
c_{q}(\<w_{x,y},\g_{0} v_{0}\>-n)
,
$$
where $c_{q}$ is the Ramanujan sum, $c_{q}(m)=\sum_{r(q)}'e_{q}(rm)$. 
By \eqref{eq:Fuc}, 
the analysis now
reduces to a classical estimate for the singular series. We may 
use the transitivity of the $\g_{0}$ sum to replace $\<w_{x,y},\g_{0} v_{0}\>$ by $\<e_{4},\g_{0} v_{0}\>$, 
extend the sum on $q$ to all natural numbers,  and use multiplicativity to write the sum
as an Euler product. Then  the resulting singular series 
$$
\fS(n)
:=
\prod_{p}
\left[
1+
\sum_{k\ge1}
{1\over[\G:\G_{0}(p^{k})]}
\sum_{\g_{0}\in\G/\G_{0}(p^{k})}
c_{p^{k}}
\bigg(
\<
e_{4},
\g_{0}\,
v_{0}
\>
-
n
\bigg)
\right]
$$ 
vanishes only on non-admissible numbers, and can easily be seen to satisfy
\be\label{eq:fSbnd}
N^{-\vep}\ll_{\vep} \fS(n)\ll_{\vep} N^{\vep},
\ee
for any $\vep>0$. See, e.g. \cite[\S4.3]{BourgainKontorovich2010}.

Next we handle the archimedean component. By our choice of $\ft$ in \eqref{eq:hatFunc}, specifically that $\hat\ft>0$ and $\hat\ft(y)>2/5$ for $|y|<1/2$, we have
$$
\fW(n)\gg {K_{0}\over N} \sum_{\ff\in\fF}
\bo_{\{
|\ff(2x,y)-n|
<
\frac N {2K_{0}}
\}}
\gg
{T^{\gd}\over N}+{T^{\gT_{1}}K_{0}\over N}
,
$$
using Lemma \ref{lem:spec2}.

Putting everything into  \eqref{eq:fMdecomp} and then into \eqref{eq:cMNUn} gives \eqref{eq:cMNbnd}, using \eqref{eq:Q0is} and \eqref{eq:TXNis}. 
\epf

Next we derive from the above that the same bound holds for $\cM_{N}^{U}$ (most of the time).

\begin{thm}\label{thm:cMNtoMNU}
There is an $\eta>0$ such that the bound \eqref{eq:cMNbnd} holds with $\cM_{N}$ replaced by $\cM_{N}^{U}$, 
except on a set 
of cardinality $\ll N^{1-\eta}$.
\end{thm}
\pf
Putting \eqref{eq:fMdecomp} into \eqref{eq:cMNUn} gives
\beann
&&
\hskip-.5in
\sum_{n<N}
|\cM_{N}(n)-\cM_{N}^{U}(n)|
\ll
\sum_{x,y\asymp X}
\sum_{n<N}|\fM(n)|
\sum_{u\mid(2x,y)\atop u\ge U}
1
\\
&\ll_{\vep}&
\sum_{y<X}
\sum_{u\mid y\atop u\ge U}
\sum_{x<X\atop 2x\equiv0(\mod u)}
\left\{
N^{\vep}
\sum_{\ff\in\fF}
{K_{0}\over N}
\left[
\sum_{n<N}
\hat\ft
\left(
(\ff(2x,y)-n)
\frac {K_{0}}N
\right)
\right]
+
K_{0}^{2}Q_{0}^{4}T^{\gT_{1}}
\right\}
\\
&\ll&
N^{\vep}
X\frac XU
T^{\gd}
,
\eeann
using \eqref{eq:fSbnd} and \eqref{eq:Q0is}. 
The rest of the argument is identical to that leading to \eqref{eq:nZbnd}.
\epf

This establishes \eqref{eq:cMNUbnd}, and hence completes our Major Arcs analysis; the rest of the paper is devoted to proving \eqref{eq:cEN2bnd}.


\newpage


\section{Minor Arcs I: Case $q<Q_{0}$}\label{sec:qQ0}

We keep all the notation of \S\ref{sec:outline}, our goal in this section being to bound \eqref{eq:IQ0K0} and \eqref{eq:IQ0}. First we return to
\eqref{eq:cRNhatUis} and reverse orders of summation, writing
\be\label{eq:RNRf}
\widehat{\cR_{N}^{U}}
(\gt)
=
\sum_{\d< U}
\mu(\d)
\sum_{\ff\in\fF}
e(-a\gt)
\hat\cR_{f,\d}
(\gt)
,
\ee
where $\ff=f-a$ according to \eqref{eq:ffVdef}, and we have set
$$
\hat\cR_{f,\d}(\gt)
:=
\sum_{ 2x\equiv0( u)}
\sum_{ y\equiv0( u)}
\gU\left(\frac {2x}X\right)
\gU\left(\frac{y}X\right)
e
\bigg(
\gt
f(2 x, y)
\bigg)
.
$$
If $u$ is even, then we have
\be\label{eq:cRffuIs}
\hat\cR_{f,\d}(\gt)
=
\sum_{ x,y\in\Z}
\gU\left(\frac {x u}X\right)
\gU\left(\frac{y u}X\right)
e
\bigg(
\gt
f( x u , y u)
\bigg)
.
\ee
If $u$ is odd, we have
$$
\hat\cR_{f,\d}(\gt)
=
\sum_{ x,y\in\Z}
\gU\left(\frac {2x u}X\right)
\gU\left(\frac{y u}X\right)
e
\bigg(
\gt
f(2 x u , y u)
\bigg)
.
$$
From now on, we focus exclusively on the case $u$ is even, the other case being handled similarly.
We first massage $\hat \cR_{f,u}$ further.

Since $f$ 
is homogeneous quadratic,
we have
$$
f( x u , y u)
=
u^{2}
f(x,y)
.
$$
Hence expressing $\gt=\frac rq+\gb$, we will 
need to write $u^{2}/q$ as a reduced fraction; to this end, introduce the notation
\be\label{eq:tilQis}
\tilde q:=(u^{2},q)\qquad
u_{0}:=u^{2}/\tilde q
,\qquad
q_{0}:=q/\tilde q
,
\ee
so that $u^{2}/q=u_{0}/q_{0}$ in lowest terms, $(u_{0},q_{0})=1$.

\begin{lem}
Recalling the notation \eqref{eq:cSfDef},
we have
\be\label{eq:cRfuIs}
\hat\cR_{f,\d }\left(\frac rq+\gb\right)
=
{1\over \d ^{2}}
\sum_{n,m\in\Z}
\cJ_{f}\left(X,\gb;{n\over \d q_{0}},{m\over \d q_{0}}\right)
\cS_{f}(q_{0},r\d _{0};n,m)
,
\ee
where we have set
\be\label{eq:cJfIs}
\cJ_{f}\left(X,\gb;{n\over uq_{0}},{m\over uq_{0}}\right)
:=
\iint\limits_{x,y\in\R}
\gU\left({x\over X}\right)
\gU\left({y\over X}\right)
e
\bigg(
\gb
f(x,y)
-{n\over uq_{0}}x-{m\over uq_{0}}y
\bigg)
dx dy
.
\ee
\end{lem}

\pf
Returning to \eqref{eq:cRffuIs}, we have
\beann
\hat\cR_{f,\d}\left(\frac rq+\gb\right)
&=&
\sum_{x,y\in\Z}
\gU\left({\d x\over X}\right)
\gU\left({\d y\over X}\right)
e_{q_{0}}
\bigg(
r
\d_{0}
f( x,y)
\bigg)
e
\bigg(
\gb
\d^{2}
f(x,y)
\bigg)
\\
&=&
\sum_{k(q_{0})}
\sum_{\ell( q_{0})}
e_{q_{0}}
\bigg(
r
\d_{0}
f(k,\ell)
\bigg)
\\
&&
\times
\left[
\sum_{x\in\Z\atop x\equiv k(q_{0})}
\sum_{y\in\Z\atop y\equiv \ell(q_{0})}
\gU\left({\d x\over X}\right)
\gU\left({\d y\over X}\right)
e
\bigg(
\gb
\d^{2}
f(x,y)
\bigg)
\right]
.
\eeann

Apply Poisson summation to the bracketed term above:
\beann
\Bigg[
\cdot
\Bigg]
&=&
\sum_{x,y\in\Z}
\gU\left({\d (q_{0}x+k)\over X}\right)
\gU\left({\d (q_{0}y+\ell)\over X}\right)
e
\bigg(
\gb
\d^{2}
f(q_{0}x+k,q_{0}y+\ell)
\bigg)
\\
&=&
\sum_{n,m\in\Z}\
\iint\limits_{x,y\in\R}
\gU\left({\d (q_{0}x+k)\over X}\right)
\gU\left({\d (q_{0}y+\ell)\over X}\right)
e
\bigg(
\gb
\d^{2}
f(q_{0}x+k,q_{0}y+\ell)
\bigg)
\\
&&
\hskip3.5in
\times
e(-nx-my)
dx dy
\\
&=&
{1\over \d^{2}q_{0}^{2}}
\sum_{n,m\in\Z}
e_{q_{0}}(nk+m\ell)
\cJ_{f}\left(X,\gb;{n\over \d q_{0}},{m\over \d q_{0}}\right),
\eeann
%
%
%
%
%
Inserting this in the above, the claim follows immediately.
\epf

We are now in position to prove the following
\begin{prop}
With the above notation,
\be\label{eq:hatRfuBnd}
\left|
\hat\cR_{f,\d }\left(\frac rq+\gb\right)
\right|
\ll
u
(\sqrt{q}|\gb|T)^{-1}
.
\ee
\end{prop}
\pf
By (non)stationary phase (see, e.g., \cite[\S8.3]{IwaniecKowalski}), the 
integral
in 
\eqref{eq:cJfIs}
 has negligible contribution unless
$$
{|n|\over \d q_{0}},
{|m|\over \d q_{0}}
\ll
|\gb|\cdot |\nabla f|\ll |\gb|\cdot TX,
$$
so the $n,m$ sum can be restricted to
\be\label{eq:nmRest}
|n|,|m|\ll  |\gb|\cdot TX\cdot \d q_{0}
\ll
 \d 
.
\ee
Here we used $|\gb|\ll (qM)^{-1}$ with $M$ given by \eqref{eq:Mis}.
In this range, stationary phase gives
\be\label{eq:cJfbnd}
\left|
\cJ_{f}\left(X,\gb;{n\over \d q_{0}},{m\over \d q_{0}}\right)\right|
\ll\min\left(X^{2},{1\over |\gb|\cdot|\operatorname{discr}(f)|^{1/2}}\right)
\ll\min\left(X^{2},{1\over |\gb| T}\right)
,
\ee
using
\eqref{eq:disc} and \eqref{eq:fFTbnd}
that
 $|\operatorname{discr}(f)|=4|B^{2}-AC|=4 a^{2}\gg T^{2}$.

Putting \eqref{eq:nmRest}, \eqref{eq:cJfbnd} and \eqref{eq:cSfbnd1} into \eqref{eq:cRfuIs}, we have
$$
\left|
\hat\cR_{f,\d }\left(\frac rq+\gb\right)
\right|
\ll
{1\over \d ^{2}}
\sum_{|n|,|m|\ll u} 
{1\over |\gb|T}
\cdot
{1\over \sqrt {q_{0}}}
,
$$
from which the claim follows, using \eqref{eq:tilQis}.
\epf

Finally, we prove the desired estimates of the strength \eqref{eq:cEN2bnd}.

\begin{thm}\label{thm:IQ0}
Recall the integrals $\cI_{Q_{0},K_{0}},\ \cI_{Q_{0}}$ from \eqref{eq:IQ0K0}, \eqref{eq:IQ0}.
There is an $\eta>0$ so that
$$
\cI_{Q_{0},K_{0}}, \ \cI_{Q_{0}}\ll N\, T^{2(\gd-1)}\, N^{-\eta}
,
$$
as $N\to\infty$.
\end{thm}

\pf
We first handle $\cI_{Q_{0},K_{0}}$.
Returning to \eqref{eq:RNRf} and applying \eqref{eq:hatRfuBnd} gives
$$
\left|
\widehat{\cR_{N}^{U}}
\left(\frac rq+\gb\right)
\right|
\ll
\sum_{\d< U}
\sum_{\ff\in\fF}
u
(\sqrt{q}|\gb|T)^{-1}
\ll
U^{2}
T^{\gd-1}
(\sqrt{q}|\gb|)^{-1}
.
$$
Inserting this into \eqref{eq:IQ0K0} and using \eqref{eq:Q0is}, \eqref{eq:fuIs} gives
\beann
\cI_{Q_{0},K_{0}}
&\ll&
\sum_{q<Q_{0}}
\sideset{}{'}\sum_{r(q)}
\int_{|\gb|<K_{0}/N}
\left|
\gb
\frac N{K_{0}}
\right|^{2}
U^{4}
T^{2(\gd-1)}
{1\over q|\gb|^{2}}
d\gb
\\
&\ll&
Q_{0}
\frac N{K_{0}}
U^{4}
T^{2(\gd-1)}
\ll
N
T^{2(\gd-1)}
N^{-\eta}
.
\eeann

Next we handle
\beann
\cI_{Q_{0}}
&\ll&
\sum_{q<Q_{0}}
\sideset{}{'}
\sum_{r(q)}
\int\limits_{{K_{0}\over N}<|\gb|<\frac1{qM}}
U^{4}
T^{2(\gd-1)}
{1\over q|\gb|^{2}}
d\gb
\\
&\ll&
Q_{0}
U^{4}
T^{2(\gd-1)}
\left({
N\over K_{0}}+{Q_{0}M}\right)
\\
&\ll&
N
T^{2(\gd-1)}
{
Q_{0}
U^{4}
\over K_{0}}
,
\eeann
which is again a power savings.
\epf


\newpage

\section{Minor Arcs II: Case $Q_{0}\le Q<X$}\label{sec:QX}

Keeping all the notation from the last section,  we now turn our attention to the integrals $\cI_{Q}$ in \eqref{eq:IQdef}.
It is no longer sufficient just to get cancellation in $\hat \cR_{f,u}$ alone, as in \eqref{eq:hatRfuBnd}; we must use the fact that $\cI_{Q}$ is an $L^{2}$-norm. 

To this end, recall the notation \eqref{eq:tilQis}, and put \eqref{eq:cRfuIs} into \eqref{eq:RNRf}, applying Cauchy-Schwarz in the $u$-variable: 
\bea
\label{eq:cRsq}
\left|\widehat{\cR_{N}^{U}}
\left(\frac rq+\gb\right)\right|^{2}
&\ll&
U
\sum_{\d< U}
\Bigg|
\sum_{\ff\in\fF}
e_{q}(-r a)
e(-a\gb)
\\
\nonumber
&&
\times
{1\over \d ^{2}}
\sum_{n,m\in\Z}
\cJ_{f}\left(X,\gb;{n\over \d q_{0}},{m\over \d q_{0}}\right)
\cS_{f}(q_{0},r\d _{0};n,m)
\Bigg|^{2}
.
\eea
Recall from \eqref{eq:ffVdef} that
$\ff=f-a$.
Insert \eqref{eq:cRsq} into \eqref{eq:IQdef} and
open the square, setting $\ff'=f'-a'$. This gives
\bea
\nonumber
\cI_{Q}
&\ll&
U
\sum_{\d< U}
{1\over \d ^{4}}
\sum_{q\asymp Q}
\sideset{}{'}
\sum_{r(q)}
\int_{|\gb|<\frac1{qM}}
\Bigg|
\sum_{\ff\in\fF}
e_{q}(-r a)
e(-a\gb)
\\
\nonumber
&&
\times
\sum_{n,m\in\Z}
\cJ_{f}\left(X,\gb;{n\over \d q_{0}},{m\over \d q_{0}}\right)
\cS_{f}(q_{0},r\d _{0};n,m)
\Bigg|^{2}
d\gb
\\
\label{eq:IQbnd1}
&
=
&
U
\sum_{\d< U}
{1\over \d ^{4}}
\sum_{n,m,n',m'\in\Z}\
\sum_{\ff,\ff'\in\fF}\
\sum_{q\asymp Q}
\\
\nonumber
&&
\times
\left[
\sideset{}{'}
\sum_{r(q)}
\cS_{f}(q_{0},r\d _{0};n,m)
\overline{\cS_{f'}(q_{0},r\d _{0};n',m')}
e_{q}(r (a'-a))
\right]
\\
\nonumber
&&
\times
\left[
\int_{|\gb|<\frac1{qM}}
\cJ_{f}\left(X,\gb;{n\over \d q_{0}},{m\over \d q_{0}}\right)
\overline{\cJ_{f'}\left(X,\gb;{n'\over \d q_{0}},{m'\over \d q_{0}}\right)}
e(\gb(a'-a))
d\gb
\right]
.
\eea

Note that again the sum has split into ``modular'' and ``archimedean'' pieces (collected in brackets, respectively), with the former being exactly equal to $\cS$ in \eqref{eq:cSdef}.

Decompose   \eqref{eq:IQbnd1} as
\be\label{eq:cIQ12}
\cI_{Q}\ll \cI_{Q}^{(=)}+\cI_{Q}^{(\neq)},
\ee
where, once $\ff$ is fixed,  we collect $\ff'$ according to whether $a'=a$ (the ``diagonal'' case) and the off-diagonal $a'\neq a$.

\begin{lem}
Assume $Q<X$. For $\square\in\{=,\neq\}$, we have
\be\label{eq:cIQ1def}
\cI_{Q}^{(\square)}
\ll
U^{6}
{X^{2}\over T}
\sum_{\ff\in\fF}
\sum_{\ff'\in\fF\atop a'\square a}
\sum_{q\asymp Q}
{
\{(a^{2},q)\cdot((a')^{2},q)\}^{1/2}(a-a',q)^{1/4}
\over
q^{5/4}
}
.
\ee
\end{lem}
\pf
Apply \eqref{eq:modBndaNeqAp} and \eqref{eq:nmRest}, \eqref{eq:cJfbnd} to \eqref{eq:IQbnd1}, giving
\beann
\cI_{Q}^{(\square)}
&\ll&
U
\sum_{u<U}
{1\over u^{4}}
\sum_{|n|,|m|,|n'|,|m'|\ll u}\
\sum_{\ff,\ff'\in\fF\atop a'\square a}\
\sum_{q\asymp Q}
\\
&&
\qquad
\times
{u^{4} 
\{(a^{2},q)\cdot((a')^{2},q)\}^{1/2}(a-a',q)^{1/4}
\over
q^{5/4}
}
\\
&&
\qquad
\qquad
\times
\int_{|\gb|<1/(qM)}
\min\left(
X^{2},
{1\over |\gb|T}
\right)^{2}
d\gb
,
\eeann
where we used \eqref{eq:tilQis}. The claim then follows immediately from \eqref{eq:Mis} and $Q<X$.
\epf

We treat $\cI_{Q}^{(=)}, \ \cI_{Q}^{(\neq)}$ separately, starting with the former; we give bounds of the quality claimed in \eqref{eq:cEN2bnd}.
\begin{prop}
There is an $\eta>0$ such that
\be\label{eq:cIQ2bnd}
\cI_{Q}^{(=)}\ll N \, T^{2(\gd-1)}N^{-\eta},
\ee
as $N\to\infty$.
\end{prop}
\pf
From \eqref{eq:cIQ1def}, we have
\beann
\cI_{Q}^{(=)}
&\ll&
U^{6}
{X^{2}\over T}
\sum_{\ff\in\fF}
\sum_{\ff'\in\fF\atop a'=a}
\sum_{q\asymp Q}
{(a^{2},q)\over q}
\\
&\ll&
{U^{6}
X^{2}\over QT}
\sum_{\ff\in\fF}
\sum_{\tilde q_{1}\mid a^{2}\atop\tilde q_{1}\ll Q}
\tilde q_{1}
\sum_{q\asymp Q\atop q\equiv0(\tilde q_{1})}
\sum_{\ff'\in\fF\atop a'=a}
1
\\
&\ll_{\vep}&
{U^{6}
X^{2}\over T}
\sum_{\ff\in\fF}
T^{\vep}
\sum_{\ff'\in\fF\atop a'=a}
1
.
\eeann
Recalling that $a=a_{\g}=\<e_{1},\g v_{0}\>$, replace the condition $a'=a$ with $a'\equiv a(\mod \lfloor Q_{0}\rfloor)$, and apply \eqref{eq:spec3}:
$$
\cI_{Q}^{(=)}
\ll_{\vep}
{U^{6}
X^{2}\over T}
T^{\gd}
T^{\vep}
{T^{\gd}\over Q_{0}^{\eta_{0}}}
.
$$
Then \eqref{eq:fuIs} and \eqref{eq:TXNis}
imply the claimed power savings.
\epf

Next we turn our attention to  $\cI_{Q}^{(\neq)}$,  the off-diagonal
 contribution. 
We decompose 
this sum
further according to whether $\gcd(a,a')$ is large or not. To this end, introduce a parameter $H$, which we will eventually set to
\be\label{eq:His}
H=U^{10/\eta_{0}}=Q_{0}{}^{\eta_{0}/10},
\ee
where, as in \eqref{eq:fuIs}, the constant $\eta_{0}>0$ comes from Lemma \ref{lem:spec3}.
Write
\be\label{eq:cIQ1break}
\cI_{Q}^{(\neq)}=\cI_{Q}^{(\neq,>)}+\cI_{Q}^{(\neq,\le)},
\ee
corresponding to whether $(a,a')>H$ or $(a,a')\le H$, respectively. We deal first with the large $\gcd$.
\begin{prop}
There is an $\eta>0$ such that
\be\label{eq:cIQ11bnd}
\cI_{Q}^{(\neq,>)}\ll N \, T^{2(\gd-1)}N^{-\eta},
\ee
as $N\to\infty$.
\end{prop}
\pf
Writing $(a,a')=h>H$, 
$\tilde q_{1}=(a^{2},q)$, 
$\tilde q_{1}'=((a')^{2},q)$, 
and
using
$(a-a',q)\le q$ in \eqref{eq:cIQ1def},
we have
\beann
\cI_{Q}^{(\neq,>)}
&\ll&
U^{6}
{X^{2}\over T}
\sum_{\ff\in\fF}
\sum_{\ff'\in\fF\atop a'\neq a,(a,a')>H}
\sum_{q\asymp Q}
{
\{(a^{2},q)\cdot((a')^{2},q)\}^{1/2}(a-a',q)^{1/4}
\over
q^{5/4}
}
\\
&\ll&
U^{6}
{X^{2}\over T}
\sum_{\ff\in\fF}
\sum_{h\mid a\atop h>H}
\sum_{\ff'\in\fF\atop a'\equiv 0(\mod h)}
\sum_{\tilde q_{1}\mid a^{2}\atop \tilde q_{1}\ll Q}
\sum_{\tilde q_{1}'\mid (a')^{2}\atop [\tilde q_{1},\tilde q_{1}']\ll Q}
(\tilde q_{1}
\tilde q_{1}')^{1/2}
\sum_{q\asymp Q\atop q\equiv 0([\tilde q_{1},\tilde q_{1}'])}
{1\over Q}
\\
&\ll_{\vep}&
U^{6}
{X^{2}\over T}
T^{\vep}
\sum_{\ff\in\fF}
\sum_{h\mid a\atop h>H}
\sum_{\ff'\in\fF\atop a'\equiv 0(\mod h)}
1
,
\eeann
where we used $[n,m]>(nm)^{1/2}$. Apply \eqref{eq:spec3} to the innermost sum, getting
\beann
\cI_{Q}^{(\neq,>)}
&\ll_{\vep}&
U^{6}
{X^{2}\over T}
T^{\vep}
T^{\gd}
{1\over H^{\eta_{0}}}
T^{\gd}
.
\eeann
By \eqref{eq:His} and \eqref{eq:fuIs}, this is a power savings, as claimed.
\epf
Finally, we handle small $\gcd$.
\begin{prop}
There is an $\eta>0$ such that
\be\label{eq:cIQ12bnd}
\cI_{Q}^{(\neq,\le)}\ll N \, T^{2(\gd-1)}N^{-\eta},
\ee
as $N\to\infty$.
\end{prop}
\pf
First note that
\beann
\cI_{Q}^{(\neq,\le)}
&=&
U^{6}
{X^{2}\over T}
\sum_{\ff\in\fF}
\sum_{\ff'\in\fF\atop a'\neq a,(a,a')\le H}
\sum_{q\asymp Q}
{
\{(a^{2},q)\cdot((a')^{2},q)\}^{1/2}(a-a',q)^{1/4}
\over
q^{5/4}
}
\\
&\ll&
U^{6}
{X^{2}\over T}
{1\over Q^{5/4}}
\sum_{\ff\in\fF}
\sum_{\ff'\in\fF\atop a'\neq a, (a,a')\le H}
\sum_{q\asymp Q}
(a,q)
(a',q)
(a-a',q)^{1/4}
.
\eeann
Write $g=(a,q)$ and $g'=(a',q)$, and let $h=(g,g')$; observe then that $h\mid (a,a')$ and $h\ll Q$. Hence we can write $g=hg_{1}$ and $g'=hg_{1}'$ so that $(g_{1},g_{1}')=1$.
Note also that $h\mid (a-a',q)$, so we can write $(a-a',q)=h\tilde g$; thus $g_{1},g_{1}',$ and $\tilde g$ are pairwise coprime, implying 
$$
[hg_{1},hg_{1}',h\tilde g]
\ge 
g_{1}g_{1}'\tilde g
.
$$
Then we have
\beann
\cI_{Q}^{(\neq,\le)}
&\ll&
U^{6}
{X^{2}\over T}
{1\over Q^{5/4}}
\sum_{\ff\in\fF}
\sum_{\ff'\in\fF\atop a'\neq a, (a,a')\le H}
\sum_{h\mid (a,a')\atop h\le H}
\sum_{g_{1}\mid a\atop g_{1}\ll Q}
\sum_{g_{1}'\mid a'\atop g_{1}'\ll Q}
\\
&&
\qquad\times
\sum_{\tilde g\mid (a-a')\atop [hg_{1},hg_{1}',h\tilde g]\ll Q}
(hg_{1})
(hg_{1}')
(h\tilde g)^{1/4}
\sum_{q\asymp Q\atop q\equiv0([hg_{1},hg_{1}',h\tilde g])}
1
\\
&\ll_{\vep}&
U^{6}
{X^{2}\over T}
{H^{9/4}\over Q^{5/4}}
\sum_{\ff,\ff'\in\fF}
T^{\vep}
\sum_{g_{1}\mid a\atop g_{1}\ll Q}
\sum_{g_{1}'\mid a'\atop g_{1}'\ll Q}
\sum_{\tilde g\mid (a-a')\atop \tilde g\ll Q}
g_{1}\,
g_{1}'\,
\tilde g^{1/4}
{Q\over g_{1}g_{1}'\tilde g}
\\
&\ll&
U^{6}
{X^{2}\over T}
{H^{9/4}\over Q^{1/4}}
\sum_{\ff\in\fF}
\sum_{\tilde g\ll Q}
{1\over \tilde g^{3/4}}
T^{\vep}
\sum_{\ff'\in\fF\atop a'\equiv a(\mod \tilde g)}
1
.
\eeann
To the last sum, we again apply Lemma \ref{lem:spec3}, giving
\beann
\cI_{Q}^{(\neq,\le)}
&\ll_{\vep}&
U^{6}
{X^{2}\over T}
{H^{9/4}\over Q^{1/4}}
T^{\gd}
\sum_{\tilde g\ll Q}
{1\over \tilde g^{3/4}}
T^{\vep}
{1\over \tilde g^{\eta_{0}}}
T^{\gd}
\ll
U^{6}
{X^{2}\over T}
{H^{9/4}\over Q_{0}^{\eta_{0}}}
T^{\gd}
T^{\vep}
T^{\gd}
,
\eeann
since $Q\ge Q_{0}$.
By \eqref{eq:His} and \eqref{eq:fuIs}, this is again a power savings, as claimed.
\epf

Putting together \eqref{eq:cIQ12}, \eqref{eq:cIQ2bnd}, \eqref{eq:cIQ1break}, \eqref{eq:cIQ11bnd}, and \eqref{eq:cIQ12bnd}, we have proved the following
\begin{thm}\label{thm:IQX}
For $Q_{0}\le Q<X$, there is some $\eta>0$ such that 
$$
\cI_{Q}\ll N\, T^{2(\gd-1)}\, N^{-\eta},
$$
as $N\to\infty$.
\end{thm}


\newpage


\section{Minor Arcs III: Case  $X\le Q<M$}\label{sec:QXT}

In this section, we continue our analysis of $\cI_{Q}$ from \eqref{eq:IQdef}, but now we need different methods to handle the very large $Q$ situation. In particular, the range of $x,y$ in \eqref{eq:cRffuIs} is now such that we have incomplete sums, so our first step is to complete them. 

To this end, recall the notation \eqref{eq:tilQis} and introduce
\be\label{eq:glfDef}
\gl_{f}\left(X,\gb;\frac n{q_{0}},\frac m{q_{0}},u\right)
:=
\sum_{x,y\in\Z}
\gU\left(\frac {ux}X\right)
\gU\left(\frac {uy}X\right)
e
\left(
-{n\over q_{0}}x-{m\over q_{0}}y
\right)
e
\bigg(
\gb
u^{2}
f(x,y)
\bigg)
,
\ee
so that, using \eqref{eq:cSfDef}, an elementary calculation gives
\be\label{eq:RfuNew}
\hat\cR_{f,u}\left(\frac rq+\gb\right)
=
\sum_{n(q_{0})}
\sum_{m(q_{0})}
\gl_{f}\left(X,\gb;\frac n{q_{0}},\frac m{q_{0}},u\right)
\cS_{f}(q_{0},ru_{0};n,m)
.
\ee
Put \eqref{eq:RfuNew} into \eqref{eq:RNRf} and apply Cauchy-Schwarz in the $u$-variable:
\bea
\label{eq:cRsq2}
\left|\widehat{\cR_{N}^{U}}
\left(\frac rq+\gb\right)\right|^{2}
&\ll&
U
\sum_{\d< U}
\Bigg|
\sum_{\ff\in\fF}
e_{q}(-r a)
e(-a\gb)
\\
\nonumber
&&
\times
\sum_{0\le n,m<q_{0}}
\gl_{f}\left(X,\gb;\frac n{q_{0}},\frac m{q_{0}},u\right)
\cS_{f}(q_{0},ru_{0};n,m)
\Bigg|^{2}
.
\eea
As before, open the square, setting $\ff'=f'-a'$, and insert the result into \eqref{eq:IQdef}:
\bea
\label{eq:IQbnd2}
\cI_{Q}
&\ll&
U
\sum_{\d< U}
\sum_{q\asymp Q}
\sum_{n,m,n',m'<q_{0}}\
\sum_{\ff,\ff'\in\fF}\
\\
\nonumber
&&
\times
\left[
\sideset{}{'}
\sum_{r(q)}
\cS_{f}(q_{0},r\d _{0};n,m)
\overline{\cS_{f'}(q_{0},r\d _{0};n',m')}
e_{q}(r (a'-a))
\right]
\\
\nonumber
&&
\times
\left[
\int_{|\gb|<1/(qM)}
\gl_{f}\left(X,\gb;\frac n{q_{0}},\frac m{q_{0}},u\right)
\overline{\gl_{f'}\left(X,\gb;\frac {n'}{q_{0}},\frac {m'}{q_{0}},u\right)}
e(\gb (a-a'))
%
%
d\gb
\right]
.
\eea
Yet again the sum has split into modular and archimedean components with the former being exactly equal to $\cS$ in \eqref{eq:cSdef}. As before, decompose 
$\cI_{Q}$ according to the diagonal ($a=a'$) and off-diagonal terms:
\be\label{eq:cIQ34}
\cI_{Q}\ll \cI_{Q}^{(=)}+\cI_{Q}^{(\neq)}.
\ee

\begin{lem}
Assume $Q\ge X$. For $\square\in\{=,\neq\}$, we have
\be\label{eq:cIQ3def}
\cI_{Q}^{(\square)}
\ll
{U
X^{3}
\over QT}
\sum_{u<U}
{1\over u^{4}}
\sum_{q\asymp Q}\
\sum_{n,m,n',m'\ll {U Q\over X}}
\sum_{\ff\in\fF}
\sum_{\ff'\in\fF\atop a'\square a}
|\cS| 
.
\ee
\end{lem}
\pf
Consider the sum $\gl_{f}$ in \eqref{eq:glfDef}. Since $x,y\asymp X/u$, 
 $|\gb|<1/(qM)$, $X\le Q$, and using \eqref{eq:Mis}, we have that
$$
|\gb u^{2} f(x,y)|\ll \frac1{QM} u^{2} T\left(\frac Xu\right)^{2}= \frac XQ\le 1.
$$
Hence 
there is contribution  only if $nx/q_{0}, my/q_{0} \ll 1$, that is,  we may restrict to the range
$$
n,m\ll u q_{0}/X
.
$$ 
In this range, we give $\gl_{f}$ the trivial bound of $X^{2}/u^{2}$. 
Putting this analysis into 
\eqref{eq:IQbnd2},
the claim follows. 
\epf
We handle the off-diagonal term first. 
\begin{prop}
Assuming $X\le Q<M$, there is some $\eta>0$ such that
\be\label{eq:cI3bnd}
\cI_{Q}^{(\neq)}
\ll
N\, T^{2(\gd-1)}\, N^{-\eta}
,
\ee
as $N\to\infty$.
\end{prop}
\pf
Since \eqref{eq:modBndaNeqAp} is such a large savings in $q>X$, we can afford to lose in the much smaller variable $T$. 
Hence put \eqref{eq:modBndaNeqAp} into \eqref{eq:cIQ3def}, estimating $(a-a',q)\le |a-a'|$ (since $a\neq a'$):
\beann
\cI_{Q}^{(\neq)}
&\ll&
{U
X^{3}
\over QT}
\sum_{u<U}
{1\over u^{4}}
\sum_{q\asymp Q}\
\sum_{n,m,n',m'\ll {UQ\over X}}
\sum_{\ff,\ff'\in\fF}
u^{4}
{
a
\cdot
a'
\over q^{5/4}
}
|a-a'|^{1/4}
\\
&\ll&
{U^{6}
X^{3}
\over 
T}
\left( {Q\over X}\right)^{4}
T^{2\gd}
{
T^{2}\over Q^{5/4}}
T^{1/4}
\\
&\ll&
U^{6}
X^{7/4}
T^{2\gd}
T^{4}
=
X^{2}T\,
T^{2(\gd-1)}
\,
\left(
U^{6}
X^{-1/4}
T^{5}
\right)
,
\eeann
where we used \eqref{eq:tilQis}, $Q<M$, and \eqref{eq:Mis}.
Using  \eqref{eq:TXNis} we have  that
\be\label{eq:gS2p}
X^{-1/4}T^{5}=N^{-59/800},
\ee
so together with \eqref{eq:fuIs}, this is clearly a substantial power savings.
\epf

Lastly, we deal with the diagonal term. We no longer save enough from $a=a'$ alone. But recall that here more cancellation can be gotten from \eqref{eq:modBndaEqAp} in the special case that $\ff(m,-n)\neq\ff'(m',-n')$.
Hence we return to 
\eqref{eq:cIQ3def} and, once $n,m,$ and $\ff$ are determined, separate $n',m'$, and $\ff'$ into cases corresponding to whether $\ff(m,-n)=\ff'(m',-n')$ or not.
Accordingly, write
\be\label{eq:cIQ312}
\cI_{Q}^{(=)}\
=\
\cI_{Q}^{(=,=)}
+
\cI_{Q}^{(=,\neq)}
.
\ee

We now estimate $\cI_{Q}^{(=,\neq)}$ using the extra cancellation in \eqref{eq:modBndaEqAp}.
\begin{prop}
Assuming $Q<XT$, there is some $\eta>0$ such that
\be\label{eq:cI42bnd}
\cI_{Q}^{(=,\neq)}
\ll
N\, T^{2(\gd-1)}\, N^{-\eta}
,
\ee
as $N\to\infty$.
\end{prop}
\pf
Returning to \eqref{eq:cIQ3def}, apply \eqref{eq:modBndaEqAp}:
\beann
\cI_{Q}^{(=,\neq)}
&\ll&
{U
X^{3}
\over QT}
\sum_{u<U}
{1\over u^{4}}
\sum_{\ff\in\fF}
\sum_{\ff'\in\fF\atop a'= a}
\sum_{q\asymp Q}\
\sum_{n,m\ll {U Q\over X}}
\sum_{n',m'\ll {U Q\over X}\atop \ff(m,-n)\neq\ff'(m',-n')}
|\cS| 
\\
&\ll&
{U
X^{3}
\over QT}
\sum_{u<U}
{1\over u^{4}}
\sum_{\ff,\ff'\in\fF}
\sum_{\tilde q_{1}\mid a^{2}\atop \tilde q_{1}\ll Q}
\sum_{q\asymp Q\atop q\equiv0(\tilde q_{1})}
\sum_{n,m,n',m'\ll {U Q\over X}}
u^{10}
{
\tilde q_{1}
\over Q^{9/8}
}
\left(T\left({UQ\over X}\right)^{2}\right)^{1/2}
\\
&\ll_{\vep}&
{U^{8}
X^{3}
\over T}
T^{2\gd}\,
T^{\vep}
\left( {U Q\over X}\right)^{4}
{
1\over Q^{9/8}
}
T^{1/2}
{UQ\over X}
\\
&\ll&
X^{2}
T
\
T^{2(\gd-1)}\,
\left(
X^{-1/8}
T^{35/8}
U^{13}
T^{\vep}
\right)
,
\eeann
where we used 
that
 $\ff(m,n)\ll T (UQ/X)^{2}$ and $Q<XT$.
 From \eqref{eq:TXNis}, we have
\be
\label{eq:gS12p}
X^{-1/8}
T^{35/8}
=
N^{-29/1600}
,
\ee
so we have again a power savings, as claimed.
\epf

Lastly, we turn to the case  $\cI_{Q}^{(=,=)}$, with $\ff(m,-n)=\ff'(m',-n')$. We exploit this condition to get savings using \eqref{eq:eqeq}.
\begin{prop}
Assuming $Q<XT$, there is some $\eta>0$ such that
\be\label{eq:cI41bnd}
\cI_{Q}^{(=,=)}
\ll
N\, T^{2(\gd-1)}\, N^{-\eta}
,
\ee
as $N\to\infty$.
\end{prop}
\pf
Returning to \eqref{eq:cIQ3def}, apply \eqref{eq:modBndaNeqAp}, and \eqref{eq:eqeq}:
\beann
\cI_{Q}^{(=,=)}
&\ll&
{U
X^{3}
\over QT}
\sum_{u<U}
{1\over u^{4}}
\sum_{q\asymp Q}\
\sum_{n,m\ll {UQ\over X}}
\sum_{\ff\in\fF}
\sum_{\ff'\in\fF\atop a'= a}
\sum_{n',m'\ll UQ/X\atop \ff(m,-n)=\ff'(m',-n')}
\hskip-.3in
u^{4}
{(a^{2},q)\over q^{5/4}}q^{1/4}
\\
&\ll&
{U
X^{3}
\over Q^{2}T}
\sum_{u<U}
\sum_{\ff\in\fF}
\sum_{\tilde q_{1}\mid a^{2}\atop\tilde q_{1}\ll Q}
\tilde q_{1}
\sum_{q\asymp Q\atop q\equiv0(\tilde q_{1})}\
\left[
\sum_{n,m\ll {UQ\over X}}
\sum_{\ff'\in\fF\atop a'= a}
\sum_{n',m'\ll UQ/X\atop \ff(m,-n)=\ff'(m',-n')}
1
\right]
\\
&\ll_{\vep}&
N^{\vep}
{U
X^{3}
\over Q^{2}T}
U
T^{\gd}
Q
\left[
\left( {UQ\over X}\right)^{2}
+
T
{UQ\over X}
\right]
\\
&\ll_{\vep}&
N^{\vep}
U^{4}
X^{2}
T^{\gd}
\ll
X^{2}T\,
T^{2(\gd-1)}\,
\left(
T^{1-\gd}
U^{4}
N^{\vep}
\right)
.
\eeann
From \eqref{eq:dim}, this is a power savings.
\epf

Combining
\eqref{eq:cIQ34}, \eqref{eq:cI3bnd}, \eqref{eq:cIQ312}, \eqref{eq:cI42bnd}, and \eqref{eq:cI41bnd}, we have the following
\begin{thm}
\label{thm:IQM}
If $X\le Q<M$, then there is some $\eta>0$ so that
$$
\cI_{Q}\ll N\, T^{2(\gd-1)} \, N^{-\eta},
$$
as $N\to\infty$.
\end{thm}

Finally, Theorems \ref{thm:IQ0}, \ref{thm:IQX},  and \ref{thm:IQM} together complete the proof of \eqref{eq:cEN2bnd}, and hence Theorem \ref{thm:Main} is proved.

\newpage

\appendix

\section{Spectral gap for the Apollonian group\\ By P\'eter P. Varj\'u}

In
 recent years some spectacular advances were made on estimating
spectral gaps (to be defined below) of
infinite co-volume subgroups of $\SL(d,\Z)$.
Bourgain and Gamburd \cite{BourgainGamburd2008} proved uniform spectral gap
estimates for Zariski-dense subgroups of
$\SL(2,\Z)$ under the additional assumption that the modulus $q$ is prime.
One of the crucial ideas in their paper is the application of
Helfgott's triple-product theorem \cite{Helfgott2008}.
The result in \cite{BourgainGamburd2008} was generalized in a series of papers
\cite{BourgainGamburd2008a}, \cite{BourgainGamburd2009},
\cite{BourgainGamburdSarnak2010}, \cite{Varju2012}, \cite{BourgainVarju2011} and \cite{SalehiVarju2011}.
Some of these require
the generalization of \cite{Helfgott2008} obtained
independently  by Breuillard, Green and Tao \cite{BreuillardGreenTao2011}
and Pyber and Szab\'o \cite{PyberSzabo2010}.

In particular, Bourgain and Varj\'u \cite[Theorem 1]{BourgainVarju2011} proved the
spectral gap for
Zariski-dense subgroups of $\SL(d, \Z)$
without any restriction for the modulus $q$.
Salehi Golsefidy and Varj\'u \cite[Theorem 1]{SalehiVarju2011} obtained the result
for Zariski-dense subgroups of perfect
arithmetic groups, but only for square-free $q$.
Unfortunately, these results do 
not cover Theorem \ref{thm:specGap}; 
the first one is not applicable to the Apollonian group, the second
one is restricted for the moduli.

In this appendix, we present an approach which differs from those
discussed above.
This is much simpler and probably would give better numerical results,
but we do not
pursue 
explicit bounds.
However, our method depends on special properties of the Apollonian
group and does not
apply to general Zariski-dense subgroups.

Recall from Section \ref{sec:prelim} that the preimage of the Apollonian
group under the homomorphism
\[
\iota :\SL(2,\C)\to \SO_F(\R)
\]
is generated by the matrices
\be\label{eq_generators1}
\pm\mat{1}{4i}{0}{1},\quad\pm\mat{2}{-i}{-i}{0},
\quad\pm\mat{2+2i}{4+3i}{-i}{-2i}.
\ee

We describe an automorphism of $\SL(2,\Z[i])$ which
transforms the above generators to matrices that will be
more convenient to work with.
Set
$A:=\mat{1}{i}{0}{1}$.
A simple calculation shows that the image of the matrices
(\ref{eq_generators1}) under the map $g\mapsto A^{-1}gA$
are
\[
\pm\mat{1}{4i}{0}{1},\quad\pm\mat{1}{0}{-i}{1},
\quad\pm\mat{1+2i}{4i}{-i}{1-2i}.
\]

We put
\be\label{eq_generators2}
\g_1=\mat{1}{4}{0}{1},\quad\g_2=\mat{1}{0}{1}{1},
\quad\g_3=\mat{1+2i}{4}{1}{1-2i}.
\ee
These are the image of (\ref{eq_generators1}) under the product of two
isomorphism: first conjugation by $A$
and then multiplication of the off-diagonal elements
by $-i$ and $i$.
We denote by $\bar \Ga$ the group generated by
$\bar S=\{\pm\g_1^{\pm1},\pm\g_2^{\pm1},\pm\g_3^{\pm1}\}$.
This is isomorphic to the group denoted by the same symbol in the paper.

First we recall two different notions of spectral gap.
The notion, ``geometric'' spectral gap, has already been explained
in Section \ref{sec:specGap}.
Recall that for an integer $q$, $\bar \Ga(q)$ denotes the kernel of
the projection map $\bar\Ga\to\SL(2,\Z[i]/(q))$.
We consider the Laplace Beltrami operator $\Delta$
on the hyperbolic orbifolds $\bar\Ga(q)\backslash\bH^3$.
We denote by $\gl_0(q)\le\gl_1(q)$ the two
smallest eigenvalues of $\Delta$ on $\bar\Ga(q)\backslash\bH^3$.
The geometric spectral gap
is an inequality of the form
$\gl_1(q)>\gl_0(q)+\vep$ for some $\vep>0$ independent
of $q$.

The other notion, ``combinatorial'' spectral gap is defined as follows.
Let $G$ be a finite group, and $S$ a symmetric set of generators.
Let $T_{G,S}$ be the Markov operator on the space $L^2(G)$
defined by
\[
T_{G,S}f(g)=\frac{1}{|S|}\sum_{\g\in S} f(\g g)
\]
for $f\in L^2(G)$ and $g\in G$.
We denote by
\[
\gl_n'(G,S)\le\ldots\le\gl_1'(G,S)\le\gl_0'(G,S)=1
\]
the eigenvalues of $T_{G,S}$ in increasing order.

The operator $Id-T_{\bar\Ga/\bar\Ga(q)}$
is a discrete analogue of the Laplacian $\Delta$
on $\bar\Ga(q)\backslash\bH^3$.
So by combinatorial spectral gap we mean the inequality
\[
\gl_1'(\bar\Ga/\bar\Ga(q),\bar S)<1-\vep
\]
for some $\vep>0$ independent of $q$.
To simplify notation, we will write
$\gl_1'(q)=\gl_1'(\bar\Ga/\bar\Ga(q),\bar S)$.

The relation between the two notions is not just an analogy.
It was shown by Brooks \cite[Theorem 1]{Brooks1986}
and Burger \cite{BurgerThesis}, \cite{Burger1986}, \cite{Burger1988} that they are equivalent
for the fundamental groups of a family of covers of
a compact manifold.
The orbifolds $\bar\Ga(q)\backslash\bH^3$ are not compact,
they even have infinite volume, however the
equivalence can be extended to cover our example,
see \cite[Theorems 1.2 and 2.1]{BourgainGamburdSarnak2011}.

We show that the congruence subgroups $\bar\Ga(q)$
of the Apollonian group have combinatorial spectral
gap which implies Theorem \ref{thm:specGap} in light of \cite[Theorems 1.2 and 2.1]{BourgainGamburdSarnak2011}.
\begin{thm}\label{th_sg}
Let $\bar\Ga$ be the Apollonian group and $\gl'_1(q)$ be as above.
There is an absolute constant $c>0$ such that $\gl_1'(q)<1-c$ for all $q$.
I.e. the Apollonian group has combinatorial spectral gap.
\end{thm}

Denote by $\Ga_1$ and $\Ga_2$ respectively, the groups
generated by $\{\g_1,\g_2\}$ and $\{\g_1,\g_3\}$ respectively.
Denote by $\bG_1$ and $\bG_2$ the Zariski-closures of $\Ga_1$
and $\Ga_2$ in $\Res_{\R|\C}\SL(2,\C)$, i.e. in $\SL(2,\C)$
considered an algebraic group over $\R$.

As we will see later, $\bG_1$ and $\bG_2$ are isomorphic to $\SL(2,\R)$.
Moreover $\Ga_1$ and $\Ga_2$ are lattices inside them.
This feature of the
Apollonian group was pointed out by Sarnak \cite{SarnakToLagarias}.
We exploit it heavily in our
approach.

Due to a result going back to
Selberg \cite{Selberg1965}, $\Ga_1$ and $\Ga_2$ have geometric spectral gaps with
respect to the congruence subgroups.
From here we can deduce the combinatorial spectral gap using
Brooks \cite[Theorem 1]{Brooks1986} (see also \cite[Theorem 1]{Brooks2007}, where
the non-compact case is considered.)

We transfer the combinatorial spectral gap property
of $\Ga_1$ and $\Ga_2$ to the Apollonian  group $\bar\Ga$
and conclude Theorem \ref{th_sg}.
This is done in following two Lemmata:

\begin{lem}\label{lm_transference}
Let $G$ be a finite group and $S\subset G$ a finite symmetric
generating set.
Let $G_1,G_2,\ldots,G_k$ be subgroups of $G$ such that
for every $g\in G$ there are $g_1\in G_1,\ldots,g_k\in G_k$
such that $g=g_1\cdots g_k$.
Then
\[
1-\gl_1'(G,S)\ge\min_{1\le i\le k}
\left\{\frac{|S\cap G_{i}|}{|S|}\cdot
\frac{1-\gl_1'(G_{i},S\cap G_{i})}{2k^2}\right\}.
\]
\end{lem}

The above Lemma and its proof below is closely related to the well-known
fact that if $G$ is generated by $S$ in $k$ steps
then one has $\gl'_1(G,S)\le1-1/|S|k^2$.
This can be found for example in \cite[Corollary 1 on page 2138]{DiaconisSaloffCoste1993}.
After circulating an earlier version of this appendix, it was pointed out to me
that an idea similar to Lemma \ref{lm_transference} has been used by
Sarnak \cite[Section 2.4]{Sarnak1990}, by Shalom
\cite{Shalom1999},  and also by
Kassabov, Lubotzky and Nikolov \cite{KLN2006}.

\begin{lem}\label{lm_generation}
Let $q\ge2$ be an integer.
Then for every $g\in \bar\Ga/\bar\Ga(q)$,
there are $g_1,\ldots, g_{10^{13}}\in\Ga_1/\Ga_1(q)$ and
$h_1,\ldots, h_{10^{13}}\in\Ga_2/\Ga_2(q)$ such that
$g=g_1h_1\cdots g_{10^{13}}h_{10^{13}}$.
\end{lem}

Lemma \ref{lm_generation} enables us to apply Lemma
\ref{lm_transference} with $k=2\cdot 10^{13}$ and
$G_i=\Ga_1/\Ga_1(q)$ for odd $i$ and
$G_i=\Ga_2/\Ga_2(q)$ for even $i$.
Now \cite{Selberg1965} and \cite[Theorem 1]{Brooks2007} provides us with 
lower bounds
on
\[
1-\gl_1'(\Ga_1/\Ga_1(q),\{\pm\g_1^{\pm 1},\pm\g_2^{\pm1}\})
\quad {\rm and}\quad
1-\gl_1'(\Ga_2/\Ga_2(q),\{\pm\g_1^{\pm 1},\pm\g_3^{\pm1}\}).
\]
Therefore Theorem \ref{th_sg} is proved once the two Lemmata are proved.

Before we proceed with the proofs, we make two remarks.
First, we note that instead of \cite{Selberg1965}
we could just as well use \cite[Theorem 1]{BourgainVarju2011}.
Second, we suggest that the constant $10^{13}$ in Lemma \ref{lm_generation}
is not optimal.
In particular, the argument we present would give 72 if the statement
is checked for $q=2^7\cdot3$, e.g. by a computer program.
Certainly there is further room for improvement but we make no
efforts to optimize the constants.

\begin{proof}[Proof of Lemma \ref{lm_transference}]
Denote by $\pi$ the
regular representation of $G$, i.e. we write
\[
\pi(g_0)f(g)=f(g_0^{-1}g)
\]
for $f\in L^2(G)$ and $g,g_0\in G$.
Let $T_{G,S}$ be the Markov operator defined above.
Let $f_0\in L^2(G)$ be an eigenfunction
with $\|f_0\|_2=1$ corresponding to $\gl_1'(G,S)$.
It is orthogonal to the constant and
\[
\langle T_{G,S}f_0,f_0\rangle=\gl_1'(G,S).
\]

Since $f_0$ is orthogonal to the constant, we have
\[
\sum_{g\in G}\langle \pi(g)f_0,f_0 \rangle=|\langle f_0,1\rangle|^2=0.
\]
Thus there is $g_0\in G$ such that
$\langle \pi(g_0)f_0,f_0 \rangle\le0$ and hence
$\|\pi(g_0)f_0-f_0\|_2\ge\sqrt2$.

By the hypothesis of the lemma, there are $g_i\in G_i$ for $1\le i\le k$
such that $g_0=g_1\cdots g_k$.
By the triangle inequality, there is some $1\le i_0\le k$
such that
\[
\|\pi(g_1\cdots g_{i_0-1})f_0-\pi(g_1\cdots g_{i_0})f_0\|_2\ge\sqrt{2}/k.
\]
Since $\pi$ is unitary, we have $\|f_0-\pi(g_{i_0})f_0\|_2\ge\sqrt{2}/k$.

We write $f_0=f_1+f_2$ such that $f_1$ is invariant under the elements
of $G_{i_0}$ in the regular representation $\pi$ and $f_2$ is orthogonal
to the space of functions invariant under $G_{i_0}$.
Then
\[
\sqrt{2}/k\le\|f_0-\pi(g_{i_0})f_0\|_2
=\|f_2-\pi(g_{i_0})f_2\|_2\le 2\|f_2\|_2.
\]
Thus $\|f_2\|_2\ge1/\sqrt{2}k$.

Now we can write
\bea
\langle T_{G,S\cap G_{i_0}}f_0,f_0\rangle
&=&\|f_1\|_2^2+\langle T_{G,S\cap G_{i_0}}f_2,f_2\rangle\nonumber\\
&\le&\|f_1\|_2^2+\gl_1'(G_{i_0},S\cap G_{i_0})\|f_2\|_2^2\nonumber\\
&=& 1-(1-\gl_1'(G_{i_0},S\cap G_{i_0}))\|f_2\|_2^2.\label{eq_TSG1}
\eea
Since
\[
T_{G,S}=\frac{|S\cap G_{i_0}|}{|S|}T_{G,S\cap G_{i_0}}
+\frac{|S\backslash G_{i_0}|}{|S|}T_{G,S\backslash G_{i_0}},
\]
we have
\be\label{eq_TSG2}
\langle T_{G,S}f_0,f_0\rangle\le1-
\frac{|S\cap G_{i_0}|}{|S|}(1-\langle T_{G,S\cap G_{i_0}}f_0,f_0\rangle).
\ee

We combine (\ref{eq_TSG1}), (\ref{eq_TSG2})
and the estimate on $\|f_2\|_2$ and get
\[
\langle T_{G,S}f_0,f_0\rangle\le1-
\frac{|S\cap G_{i_0}|}{|S|}\cdot
\frac{1-\gl_1'(G_{i_0},S\cap G_{i_0})}{2k^2}
\]
which was to be proved.
\end{proof}

Now we turn to the proof of Lemma \ref{lm_generation}.
It will be convenient to write
\[
A_k(q)=\{g_1h_1\cdots g_kh_k:g_1,\ldots g_k\in \Ga_1/\Ga_1(q),
h_1,\ldots h_k\in \Ga_2/\Ga_2(q)\}
\]
First we consider the case when $q$ is the power of a prime;
the general case will be easy to deduce from this.

\begin{lem}\label{lm_local}
Let $p$ be a prime and $m$ a positive integer.
Then $A_{10^{13}}(p^m)=\bar\Ga/\bar\Ga(p^m)$.
\end{lem}

We use different methods when $p$ is 2 or 3 compared to when it is larger.
First we consider the latter situation.
\begin{proof}[Proof of Lemma \ref{lm_local} for $p\ge5$]
It is well-known and easy to check that the group generated by $\g_1$ and
$\g_2$ is
\be\label{eq_Gamma1}
\Ga_1=\left\{\mat{a}{b}{c}{d}\in\SL(2,\Z): b\equiv0\; \mod 4\right\}.
\ee
Thus $\Ga_1/\Ga_1(p^m)=\SL(2,\Z/p^m\Z)$ for $p\neq2$.

By simple calculation:
\[
\mat{a^{-1}}{0}{0}{a}\mat{\frac{1}{2}}{0}{\frac{-1}{8}}{2}
\g_3^2\mat{1}{0}{\frac{1}{8}}{1}\g_3^{-1}
\mat{a}{0}{0}{a^{-1}}
=\mat{1}{0}{\frac{-3ia^2}{2}}{1}.
\]
Since $p\neq2$ we can divide by $2$ in the ring $\Z/p^m\Z$, hence for $(a,p)=1$, the
matrices in the above calculation are in $\Ga_1/\Ga_1(p^m)$ except for
$\g_3$.
Therefore
\[
\mat{1}{0}{\frac{-3ia^2}{2}}{1}\in A_3(p^m).
\]

Using this, we want to show that
\begin{equation}\label{eq:unip}
\mat{1}{0}{a i}{1}\in A_{12}(p^m)
\end{equation}
for all $a\in\Z/p^m\Z$.
To do this, we need to show that for every element $x\in\Z/p^m\Z$,
we can find elements $a_1,\ldots, a_k\in\Z/p^m\Z$ for some $0\le k\le 4$,
such that $a_1,\ldots, a_k$ are not divisible by $p$ and $x=a_1^2+\ldots+a_k^2$.
If $m=1$, this simply follows from the fact that any positive integer
is a sum of at most 4 squares,
and the $a_i$ can not be divisible by $p$ since $0<a_i\le x\le p$ and
at least one of the inequalities are
strict.

Suppose that $m>1$, $x\in\Z/p^m\Z$ and $a_1^2+\ldots+a_k^2\equiv x\;\rm mod\;p$
with none of $a_1\ldots a_k$ divisible by $p$.
Then by Hensel's lemma (recall that $p\neq2$), there is an
$a_1'\in\Z/p^m\Z$ such that
\[
(a_1')^2=a_1^2+(x-a_1^2-\ldots-a_k^2).
\]
This proves the claim for arbitrary $m\ge1$.

Multiplying (\ref{eq:unip}) by a suitable unipotent element of
$\Ga_1/\Ga_1(p^m)$, we
can get
\[
\mat{1}{0}{a}{1}\in A_{12}(p^m)
\]
for $a\in\Z[i]/(p^m)$.
We can prove the same for the upper triangular unipotents by a
very similar argument.

Again, by simple calculation:
\[
\mat{1}{a}{0}{1}\mat{1}{0}{b}{1}\mat{1}{c}{0}{1}
=\mat{1+ab}{a+c+abc}{b}{1+bc}.
\]
This shows that
\[
\mat{a'}{b'}{c'}{d'}\in A_{36}(p^m)
\]
for all $a',b',c',d'\in\Z[i]/(p^m)$, $a'd'-b'c'=1$, provided $c'$ is not divisible
by a prime above $p$.

Thus, $A_{36}(p^m)$ contains more than half of the group
$\bar\Ga/\bar\Ga(p^m)$, hence
\[
A_{72}(p^m)=\bar\Ga/\bar\Ga(p^m).
\]
\end{proof}

\begin{proof}[Proof of Lemma \ref{lm_local} for $p=2$ and 3]
We give the proof for $p=2$ and then explain the differences
for $p=3$.

We prove by induction the following statement.
For every $m\ge7$ and $g\in\bar\Ga(2^7)/\bar\Ga(2^m)$,
there are $g_1,g_2,g_3\in \Ga_1(2^{2})/\Ga_1(2^m)$ such that
\[
g=g_1\g_3g_2\g_3^{-1}\g_3^2 g_3\g_3^{-2}.
\]

For $m=7$ this is clear since we can take $g_1=g_2=g_3=1$.
Now assume that $m>7$ and the statement holds for $m-1$.
In this proof, we denote by $1$ the multiplicative unit (identity matrix)
and by $0$ the matrix with all entries 0.
Let $g\in\bar\Ga(2^7)/\bar\Ga(2^m)$ be arbitrary.
By the induction hypothesis, there is $h_1,h_2,h_3\in \Ga_1(2^{2})/\Ga_1(2^m)$
such that
\[
g-h_1\g_3h_2\g_3^{-1}\g_3^2 h_3\g_3^{-2}=2^{m-1}x,
\]
where $x$ can be considered as an element of $\Mat(2,\Z[i]/(2))$, i.e.
a $2\times 2$ matrix with elements in $\Z[i]/(2)$.
Since $g,h_1,h_2,h_3$ has determinant 1 and congruent to the unit element
mod 2, $x$ has trace 0.

Now we look for suitable $x_1,x_2,x_3\in\Mat(2,\Z)$ such that
\[
x_1+\g_3x_2\g_3^{-1}+\g_3^2 x_3\g_3^{-2}\equiv 2^{m-1}x\; \mod 2^m.
\]
Moreover, we 
ensure
that $x_i\equiv 0\;\mod 2^{m-4}$ and that
$\Tr(x_i)\equiv 0\;\mod 2^{m}$ for all $i=1,2,3$.
Since $m\ge8$, this implies that $h_i+x_i\equiv 1\;\mod 4$ and 
$\det(h_i+x_i)\equiv1 \;\mod 2^m$, hence $h_i+x_i\in \Ga_1(2^{2})/\Ga_1(2^m)$.
Recall (\ref{eq_Gamma1}) from the previous proof.
If the matrices $x_i$ satisfy the claimed properties then
\bean
(h_1+x_1)\g_3(h_2+x_2)\g_3^{-1}\g_3^2 (h_3+x_3)\g_3^{-2}\\
\equiv
h_1\g_3h_2\g_3^{-1}\g_3^2 h_3\g_3^{-2}+
x_1+\g_3x_2\g_3^{-1}+\g_3^2 x_3\g_3^{-2}\equiv g\; \mod 2^m.
\eean

The matrices $x_1,x_2,x_3$ can be chosen to be a suitable linear
combination of the matrices in the following calculations, and this
finishes the induction:
\[
2^{m-1}\mat{0}{1}{0}{0}+\g_3
0\g_3^{-1}+\g_3^2
0\g_3^{-2}
\equiv 2^{m-1}\mat{0}{1}{0}{0} \mod 2^m,
\]
\[
2^{m-1}\mat{0}{0}{1}{0}+\g_3
0\g_3^{-1}+\g_3^2
0\g_3^{-2}
2^{m-1}\equiv \mat{0}{0}{1}{0} \mod 2^m,
\]
\[
2^{m-1}\mat{1}{0}{0}{-1}+\g_3
0\g_3^{-1}+\g_3^2
0\g_3^{-2}
\equiv 2^{m-1}\mat{1}{0}{0}{-1} \mod 2^m,
\]
\[
2^{m-2}\mat{1}{3}{1}{-1}+
\g_3 2^{m-2}\mat{0}{1}{0}{0}\g_3^{-1}+\g_3^2
0\g_3^{-2}
\equiv 2^{m-1}\mat{-i}{0}{0}{i} \mod 2^m,
\]
\[
2^{m-3}\mat{-4}{0}{3}{4}+\g_3
2^{m-3}\mat{0}{0}{1}{0}\g_3^{-1}+\g_3^2
0\g_3^{-2}
\equiv 2^{m-1}\mat{0}{0}{i}{0} \mod 2^m,
\]
\[
2^{m-4}\mat{2}{15}{4}{-2}+\g_3
0\g_3^{-1}+\g_3^2
2^{m-4}\mat{0}{1}{0}{0}\g_3^{-2}
\equiv 2^{m-1}\mat{-i}{i}{0}{i} \mod 2^m.
\]

Now we showed that
\[
A_3(2^{m})\supseteq\bar\Ga(2^7)/\bar\Ga(2^m).
\]
The index of $\bar\Ga(2^7)/\bar\Ga(2^m)$ in
$\bar\Ga/\bar\Ga(2^m)$ is at most
\[
|\SL(2,\Z[i]/(2^7))|=46\cdot64^6.
\]
This shows that
\[
A_{10^{13}}(2^{m})=\bar\Ga/\bar\Ga(2^m).
\]

Now we turn to the case $p=3$.
By the same argument, one can show that
for every $m\ge1$ and $g\in\bar\Ga(3)/\bar\Ga(3^m)$,
there are $g_1,g_2,g_3\in \Ga_1/\Ga_1(3^m)$ such that
\[
g=g_1\g_3g_2\g_3^{-1}\g_3^2 g_3\g_3^{-2}.
\]
The only significant difference is that one needs to use the
following identities:
\[
3^{m-1}\mat{1}{3}{1}{-1}+
\g_33^{m-1}\mat{0}{1}{0}{0}\g_3^{-1}+\g_3^2
0\g_3^{-2}
\equiv 3^{m-1}\mat{i}{i}{0}{-i} \mod 3^m,
\]
\[
3^{m-1}\mat{-4}{16}{3}{4}+\g_33^{m-1}
\mat{0}{0}{1}{0}\g_3^{-1}+\g_3^2
0\g_3^{-2}
\equiv 3^{m-1}\mat{i}{0}{-i}{-i} \mod 3^m,
\]
\[
3^{m-1}\mat{2}{15}{4}{-2}+\g_3
0\g_3^{-1}+\g_3^23^{m-1}
\mat{0}{1}{0}{0}\g_3^{-2}
\equiv 3^{m-1}\mat{i}{-i}{0}{-i} \mod 3^m.
\]

Using this claim, one can finish the proof as above.
\end{proof}

\begin{proof}[Proof of Lemma \ref{lm_generation}]

Let $q$ be an integer and $q=p_1^{m_1}\cdots p_n^{m_n}$ where $p_i$
are primes.
We prove that
\[
 A_{10^{13}}(q)=A_{10^{13}}(p_1^{m_1})\times\ldots\times A_{10^{13}}(p_n^{m_n}).
\]
Let $x\in A_{10^{13}}(p_1^{m_1})\times\ldots\times A_{10^{13}}(p_n^{m_n})$
be arbitrary.
By definition, for each $k$, we can find elements $g_1^{(k)},\ldots
g_{10^{13}}^{(k)}\in \Ga_1/\Ga_1(q)$
and $h_1^{(k)},\ldots h_{10^{13}}^{(k)}\in \Ga_2/\Ga_2(q)$
such that
\[
x\equiv g_1^{(k)}h_1^{(k)}\cdots
g_{10^{13}}^{(k)}h_{10^{13}}^{(k)}\;{\rm mod}\; p_k^{m_k}.
\]

Since $\Ga_1/\Ga_1(p^m)$ and $\Ga_2/\Ga_2(p^m)$ are the direct product
of local factors, we can
find elements $g_1,\ldots, g_{10^{13}}\in\Ga_1/\Ga_1(p^m)$ and
$h_1,\ldots, h_{10^{13}}\in\Ga_2/\Ga_2(p^m)$
such that
\[
g_i\equiv g_i^{(k)}\;{\rm mod}\; p_k^{m_k}\quad{\rm and}\quad
h_i\equiv h_i^{(k)}\;{\rm mod}\; p_k^{m_k}
\]
for each $i$ and $k$.
Thus
\[
x= g_1h_1\cdots g_{10^{13}}h_{10^{13}}\in A_{10^{13}}(q).
\]

Using Lemma \ref{lm_local} 
we get
\[
\bar\Ga/\bar\Ga(q)
\supset A_{10^{13}}(q)
\supset A_{10^{13}}(p_1^{m_1})\times\ldots\times A_{10^{13}}(p_n^{m_n})
=\bar\Ga/\bar\Ga(p_1^{m_1})\times\ldots\times \bar\Ga/\bar\Ga(p_{n}^{m_n}).
\]
Obviously
\[
\bar\Ga/\bar\Ga(q)\subset \bar\Ga/\bar\Ga(p_1^{m_1})\times\ldots\times
\bar\Ga/\bar\Ga(p_{n}^{m_n})
\]
hence all these containments must be equality.
\end{proof}

\newpage

\bibliographystyle{alpha}

\bibliography{AKbibliog}

\end{document}